\documentclass[a4paper,11pt]{amsart}
\usepackage[utf8]{inputenc}
\usepackage[left=2.7cm,right=2.7cm,top=3.5cm,bottom=3cm]{geometry}
\usepackage{amsthm,amssymb,amsmath,amsfonts,mathrsfs,amscd,yfonts}
\usepackage{bbm}
\usepackage{bm}
\usepackage[dvipsnames]{xcolor}
\usepackage{tikz-cd}
\usepackage[hidelinks]{hyperref}
\usepackage{latexsym}
\usepackage[english]{babel}
\usepackage{newlfont}
\usepackage{mathtools,braket}
\usepackage[all]{xy}
\usepackage{booktabs}
\usepackage{stmaryrd}
\usepackage{comment}
\usepackage{verbatim}
\usepackage{turnstile}
\usepackage{dsfont}
\usepackage{enumitem}

\newcommand{\bb}{\mathbb}
\newcommand{\frk}{\mathfrak}
\newcommand{\cl}{\mathcal}

\newcommand{\kap}{{\kappa_{\psi,g}}}
\newcommand{\kapone}{{\kappa_{\psi,g,1}}}

\newcommand{\kapinftys}{{\kappa_{\psi,g,\infty}}}
\newcommand{\kapinftyones}{{\kappa_{\psi,g,1,\infty}}}
\newcommand{\kapinftyns}{{\kappa_{\psi,g,n,\infty}}}
\newcommand{\kapinftynqs}{{\kappa_{\psi,g,nq,\infty}}}

\theoremstyle{plain}

\newtheorem{thmintro}{Theorem}

\newtheorem{theorem}{Theorem}[section]
\newtheorem{lemma}[theorem]{Lemma}
\newtheorem{prop}[theorem]{Proposition}

\theoremstyle{definition}
\newtheorem{defn}[theorem]{Definition}

\theoremstyle{remark}
\newtheorem{rk}{Remark}

\newcommand{\mbf}{\mathbf}
\newcommand{\mrm}{\mathrm}

\newcommand{\scr}{\mathscr}

\newcommand{\hf}{\mathscr{F}}
\newcommand{\hg}{\mathscr{G}}

\newcommand{\Gal}{\mathrm{Gal}}
\newcommand{\ch}{\mathrm{CH}}

\newcommand{\longmono}{\mbox{$\;\lhook\joinrel\longrightarrow\;$}}

\DeclareMathOperator{\ab}{ab}
\DeclareMathOperator{\ac}{ac}

\DeclareMathOperator{\As}{As}
\DeclareMathOperator{\bal}{bal}

\DeclareMathOperator{\CH}{CH}
\DeclareMathOperator{\Char}{Char}

\DeclareMathOperator{\cores}{cor}

\DeclareMathOperator{\cyc}{cyc}

\DeclareMathOperator{\et}{\acute{e}t}
\DeclareMathOperator{\Fil}{Fil}
\DeclareMathOperator{\Fr}{Fr}

\DeclareMathOperator{\GL}{GL}
\DeclareMathOperator{\Gr}{Gr}

\DeclareMathOperator{\Hom}{Hom}
\DeclareMathOperator{\id}{id}

\DeclareMathOperator{\Ind}{Ind}
\DeclareMathOperator{\Iw}{Iw}

\DeclareMathOperator{\nr}{nr}
\DeclareMathOperator{\ord}{ord}

\DeclareMathOperator{\Sel}{Sel}
\DeclareMathOperator{\SL}{SL}

\DeclareMathOperator{\std}{std}

\DeclareMathOperator{\tors}{tors}

\begin{document}

\title[An anticyclotomic Euler system of Hirzebruch--Zagier cycles I]{An anticyclotomic Euler system of Hirzebruch--Zagier cycles I: Norm relations and $p$-adic interpolation}

\author{Ra\'ul Alonso}
\author{Francesc Castella} 
\author{\'Oscar Rivero}

\date{\today}

\address{R. A.: Department of Mathematics, University of California, Santa Barbara, CA 93106, USA}
\email{raular@ucsb.edu}

\address{F. C.: Department of Mathematics, University of California, Santa Barbara, CA 93106, USA}
\email{castella@ucsb.edu}


\address{O. R.: Departamento de Matem\'aticas, Universidade de Santiago de Compostela, Spain}
\email{oscar.rivero@usc.es}

\maketitle

\begin{abstract}
We construct an anticyclotomic Euler system for the Asai Galois representation associated to $p$-ordinary Hilbert modular forms over real quadratic fields. We also show that our Euler system classes vary in $p$-adic Hida
families. The construction is based on the study of certain Hirzebruch--Zagier cycles obtained from modular curves of varying level diagonally emdedded into the product with a Hilbert modular surface. By Kolyvagin's methods, in the form developed by Jetchev--Nekov\'{a}\v{r}--Skinner in the anticyclotomic setting, the construction yields new applications to the Bloch--Kato conjecture and the Iwasawa Main Conjecture.
%
\end{abstract}

\setcounter{tocdepth}{1}
\tableofcontents

\section{Introduction}

In this paper, we construct an anticyclotomic Euler system for the Asai Galois representation associated to $p$-ordinary Hilbert modular forms over real quadratic fields. We also show that our Euler system classes vary in $p$-adic families. 
Using Kolyvagin's methods, in the form developed by Jetchev--Nekov\'{a}\v{r}--Skinner \cite{JNS} in the anticyclotomic setting, our construction yields new  applications towards the Bloch--Kato conjecture in rank one and towards the Iwasawa Main Conjecture.




\subsection{Setting}

Let $F$ be a real quadratic field with ring of integers $\cal{O}_F$. With conventions as in Section~\ref{sec:review}, we let $g\in S_l(\mathfrak{N}_g,\chi_g;\bb{C})$ be a Hilbert eigenform over $F$ of level $\frk{N}_g\subset\cal{O}_F$, parallel weight $l\geq 2$, and nebentype $\chi_g$.  Let $\cal{K}$ be an imaginary quadratic field of discriminant $-D<0$, and let $\psi$ be a Hecke character of $\cal{K}$ of conductor $\frk{c}\subset\cal{O}_{\cal{K}}$ and infinity type $(1-k,0)$ for some even integer $k\geq 2$. We assume the self-duality condition
\begin{equation}\label{eq:sd}
\chi_\psi\varepsilon_{\cal{K}}\chi_g\vert_{\bb{A}_{\bb{Q},f}^\times}=1,\nonumber
\end{equation}
where $\chi_\psi$ denotes the central character of $\psi$, and $\varepsilon_{\cal{K}}$ is the quadratic character corresponding to $\cal{K}/\bb{Q}$. Put $N=N_{\cl{K}/\bb{Q}}(\frk{c})DN_{F/\bb{Q}}(\frk{N}_g)$. Fix a prime $p\nmid 2N$ and an embedding $\iota_p:\overline{\mathbb Q}\hookrightarrow\overline{\mathbb Q}_p$, where $\overline{\bb{Q}}$ denotes the algebraic closure of $\bb{Q}$ in $\bb{C}$, and let $E=L_{\frk{P}}$ be a finite extension of $\bb{Q}_p$ with ring of integers $\cal{O}$ containing the image under $\iota_p$ of
the Fourier coefficients of $g$ and the values of $\psi$. Let $\rho_g:G_F\rightarrow\mathrm{GL}(V_g)\simeq\mathrm{GL}_2(E)$ be the $p$-adic Galois representation attached to $g$, and let
\[
\mathrm{As}(V_g):=\otimes\mbox{-}\mathrm{Ind}_F^\bb{Q}(V_g)
\]
be the associated four-dimensional \emph{Asai representation}. 

\subsection{The Euler system} 

A main result of this paper is the construction of an anticyclotomic Euler system 
for the conjugate self-dual  $G_{\cal{K}}$-representation
\[
V:=\mathrm{As}(V_g)\vert_{G_{\cal{K}}}(\psi_{\frk{P}}^{-1})(2-l-k/2),
\]
where $\psi_{\frk{P}}:G_{\cal{K}}\rightarrow L^\times$ is the $p$-adic avatar of $\psi$. 

For each positive integer $n$, we let $\cl{K}[n]$ denote the maximal $p$-extension inside the ring class field of $\cl{K}$ of conductor $n$, and let $\mathcal{S}$ be the set of all squarefree products of primes $q$ split in both $F$ and $\cal{K}$ and coprime to $pN$. For any prime $\mathfrak {q}$ of $\cl{K}$ where $V$ is unramified, put
\[
P_{\mathfrak{q}}(V;X)=\det(1-\Fr_{\frk{q}}^{-1}X\vert V),
\]
where $\Fr_\frk{q}$ denotes an arithmetic Frobenius element at $\frk{q}$. The representation $\mathrm{As}(V_g)$ appears in the middle-degree cohomology of Hilbert modular surfaces, and so taking integral coefficients we obtain a $G_{\cal{K}}$-stable $\cal{O}$-lattice $T\subset V$. 
Let $h_{\cal{K}}$ denote the class number of $\cal{K}$, and let  
\[
\mathrm{Sel}_{\bal}(\cal{K}[np^\infty],T)\subset H^1_{\Iw}(\cal{K}[np^\infty],T):=\varprojlim_s H^1(\cal{K}[np^s],T)
\]
be the \emph{balanced Selmer group} introduced in Section~\ref{subsec:Selmer}; in the range $k<2l$, this interpolates the Bloch--Kato Selmer groups for $T$ over the tower $\cal{K}[np^\infty]/\cal{K}[n]$. 

\begin{thmintro}[Theorem~\ref{thm:ES}]\label{thmA}
Suppose $p$ splits in $\cal{K}$ and $p\nmid h_{\cal{K}}$, and that $g$ is ordinary at $p$. 
There exists a collection of classes
\[
\left\lbrace\kapinftyns\in \mathrm{Sel}_{\bal}(\cal{K}[np^\infty],T)\;\colon\; m\in\mathcal{S}\right\rbrace
\]
such that whenever $n, nq\in\mathcal{S}$ with $q$ a prime, we have
\begin{equation}
{\cores}_{\cl{K}[nq]/\cl{K}[n]}(\kapinftynqs)= P_{\frk{q}}(V;{\Fr}_{\frk{q}}^{-1})\,\kapinftyns,\nonumber
\end{equation}
where $\mathfrak{q}$ is a prime of $\cl{K}$ above $q$.
\end{thmintro}

It follows from our construction that the Euler system classes of Theorem~\ref{thmA} are interpolated $p$-adically along the parallel weight Hida family $\hg$ passing through the ordinary $p$-stabilization of $g$. Indeed, the results of Section~\ref{sec:cycles} yield a construction of two-variable families of cohomology classes
\[
\kappa_\infty(\hf,\hg) \in H^1(\mathbb Q, \mathbb V_{\hf \hg}^{\dag})
\]
where $\hf$ (resp. $\hg$) is a Hida family of elliptic modular forms (parallel weight Hilbert modular forms), and $\mathbb V_{\hf \hg}^{\dag}$ is a self-dual twist of the tensor product $\mathbb V_{\hf}\otimes\mathrm{As}(\mathbb V_{\hg})$ for the standard big Galois representations $\mathbb V_{\hf}$ and $\mathbb V_{\hg}$ associated to $\hf$ and $\hg$, respectively. The construction of the classes $\kappa_\infty(\hf,\hg)$ is based on the study of the $p$-adic \'etale Abel--Jacobi image of certain generalized Hirzebruch--Zagier cycles obtained from modular curves of varying $p$-power level diagonally embedded into the product with Hilbert modular surfaces, and is inspired by the approach by Darmon--Rotger \cite{DR3} 
in the setting of triple products of elliptic modular forms. 

The classes $\kapinftyns$ of Theorem~\ref{thmA} are obtained from a variant of the classes $\kappa_\infty(\hf,\hg)$ with added tame level at $n$ by taking $\hf$ to be a CM Hida family attached to $\psi$  
and $\hg$ to be the parallel weight Hida family passing through $g$. The bulk of the work then goes into the proof of the tame Euler system norm relations, which is obtained by a careful study of the image under various degeneracy maps of the generalized  Hirzebruch--Zagier cycles used in the construction. It would be interesting to study whether an alternative proof of our Euler systems norm relations can be given following the automorphic approach of Loeffler--Skinner--Zerbes \cite{LSZ-JEMS}, as done by Grossi \cite{Grossi-AF} in the related setting of Asai--Flach classes. 

\subsection{Applications}

By Kolyvagin's Euler system machinery, in the form developed by Jetchev--Nekov\'{a}\v{r}--Skinner \cite{JNS} in the anticyclotomic setting, we deduce bounds on Selmer groups from the non-triviality of our Euler system. Our first result in this direction gives new evidence towards the Bloch--Kato conjecture \cite{BK} in rank one. 

In the next theorem, we use `big image' to refer to Hypothesis~(HS) in  Section~\ref{sec:big-image}; in Proposition~\ref{prop:suff} we shall give sufficient conditions under which this is satisfied. As a piece of notation, write 
\[
\kappa_{\psi,g,n}\in\mathrm{Sel}_{\bal}(\cal{K}[n],T)
\]
for the image of 
$\kappa_{\psi,g,n,\infty}$ under the corestriction $H^1_{\Iw}(\cl{K}[np^\infty],T) \rightarrow H^1(\cl{K}[n],T)$. 



\begin{thmintro}[Theorem~\ref{thm:BK}]\label{thmB}
Let the setting and hypotheses be as in Theorem~\ref{thmA}, and assume in addition that $V$ has big image.
Let $\kap:=\kapone$. If $k<2l$, then the following implication holds:
\[
\kap \neq 0 \quad \Longrightarrow \quad \dim_E\mathrm{Sel}(\cl{K},V)=1,
\]
where $\Sel(\cl{K},V)\subset H^1(G_{\cl{K}},V)$ is the Bloch--Kato Selmer group.
\end{thmintro}

We also deduce applications to the Iwasawa Main Conjecture for $V$. More precisely, let
\[
\mathrm{Sel}_{\bal}(\cal{K}[np^\infty],A)\subset H^1(\cal{K}[np^\infty],A)
\]
be the balanced Selmer group dual to $\mathrm{Sel}_{\bal}(\cal{K}[np^\infty],T)$, where $A=\mathrm{Hom}_{\bb{Z}_p}(T,\mu_{p^\infty})$, and let $X_{\bal}(\cal{K}[np^\infty],A)$ 
denote its Pontryagin dual. Write $p\cal{O}_{\cal{K}}=\frk{p}\overline{\frk{p}}$, with $\frk{p}$ the prime of $\cal{K}$ above $p$ induced by $\iota_p$, and put 
$\Lambda_{\hf}=\cal{O}[\![\Gamma_{\frk{p}}]\!]$, where $\Gamma_{\frk{p}}$ is the Galois group of the unique $\bb{Z}_p$-extension of $\cal{K}$ unramified outside $\frk{p}$. As explained in Section~\ref{sec:norm-relations}, this can be naturally identified with the Iwasawa algebra for the anticyclotomic $\bb{Z}_p$-extension $\cal{K}[p^\infty]/\cal{K}$.

\begin{thmintro}[Theorem~\ref{thm:IMC}]\label{thmC}
Let the setting and hypotheses be as in Theorem~\ref{thmA}, and assume in addition that $V$ has big image. If the class $\kapinftys:=\kapinftyones$ is not $\Lambda_{\hf}$-torsion, then the modules $\Sel_{\bal}(\cal{K}[p^\infty],T)$ and $X_{\bal}(\cal{K}[p^\infty],A)$ have both $\Lambda_{\hf}$-rank one, and 
\[
\Char_{\Lambda_{\hf}}(X_{\bal}(\cal{K}[p^\infty],A)_{\tors})\supset \Char_{\Lambda_{\hf}}\biggl(\frac{\Sel_{\bal}(\cal{K}[p^\infty],T)}{\Lambda_{\hf} \cdot\kapinftys} \biggr)^2
\]
as ideals in $\Lambda_{\hf}\otimes_{\bb{Z}_p}\bb{Q}_p$, where the subscript $\mathrm{tors}$ denotes the $\Lambda_{\hf}$-torsion submodule.
\end{thmintro}

\subsection{Further directions}

We expect the classes $\kappa_\infty(\hf,\hg)$ to satisfy an explicit reciprocity law relating their images under a Perrin-Riou big logarithm map to a $p$-adic Asai $L$-function as constructed by Ishikawa \cite{Ishikawa} 
in the $p$ inert case. In particular, this would allow us to deduce a variant of Theorem~\ref{thmC} for the $\mathcal{F}$-unbalanced Selmer group $X_{\mathcal F}(\cal{K}[p^\infty],A)$ of Section~\ref{subsec:Selmer}, as well as applications to the Bloch--Kato conjecture for $V$ in analytic rank $0$ in the range $k\geq 2l$. We plan to study these problems in a sequel to this paper. 

It would also be interesting to study the implications of our Euler system in the ``degenerate'' case where $g=\mathrm{BC}(g_\circ)$ is the base-change of an elliptic modular form $g_\circ$. As observed in e.g. \cite[\S{1.3}]{Liu}, in this case the Asai representation $\mathrm{As}(V_g)$ contains $\mathrm{Sym}^2(V_{g_\circ})$ as a direct summand; following the approach in \cite{LZ-ad}, one might hope to build on this decomposition and (a suitable variant of) the Euler system of Theorem~\ref{thmA} to improve on the results obtained \cite{ACR2} towards the anticyclotomic Iwasawa Main Conjecture for the adjoint representation of $g_\circ$.



\subsection{Related works}

The Euler system constructed in this paper may be seen as an anticyclotomic analogue of the Asai--Flach Euler system of Lei--Loeffler--Zerbes \cite{LLZ-AF}. Generalized Hirzebruch--Zagier cycles similar to those constructed in this paper were studied in \cite{BF} and \cite{FJ}; following the approaches of \cite{DR1} and \cite{DR2}, respectively, these works express the syntomic Abel--Jacobi images of these cycles in terms of special values of $p$-adic Asai $L$-functions outside the range of interpolation. (Unfortunately, the results of \cite{BF} restrict to $g$ of non-parallel weight, while \cite{FJ} is subject to a conjecture on the integral cohomology of Hilbert modular surfaces.) 

In a different direction, 
Kolyvagin's methods, in a form influenced by the work of Bertolini--Darmon \cite{BD2} using level-raising congruences, were first applied in the setting of Hirzebruch--Zagier cycles in work of Y.~Liu \cite{Liu}, with applications towards the Bloch--Kato conjecture analogous to our Theorem~\ref{thmB} in the case of parallel weight $2$ (but allowing more general twisted triple products), in addition to certain counterparts in analytic rank $0$.





\subsection{Acknowledgements}

We heartily thank Michele Fornea for stimulating discussions on the construction of our Euler system classes and his detailed explanations of his earlier results on generalized Hirzebruch--Zagier cycles. We also thank David Loeffler, Chris Skinner, and Shou-Wu Zhang for several useful exchanges related to this work. 

During the preparation of this paper, F.C. was partially supported by the NSF grants DMS-2101458 and DMS-2401321; O.R. was supported by PID2023-148398NA-I00, funded by MCIU/AEI/10.13039/501100011033/FEDER, UE. 

This material is based upon work supported by the NSF grant DMS-1928930 while the  authors were in residence at MSRI/SLMath during the spring of 2023 for the program ``Algebraic Cycles, $L$-Values, and Euler Systems''.


\section{Review of Hilbert cuspforms}\label{sec:review}

In this section we summarize our conventions for Hilbert modular forms, which are mainly taken from \cite{FJ} with some minor changes. Since for the purpose of this article we are only interested in Hilbert modular forms of parallel weight, we restrict the discussion to this case.

\subsection{General notations}

Let $F$ be a totally real number field with ring of integers $\cal{O}_F$. Let $\mrm{I}_F$ denote the set of field embeddings of $F$ into $\overline{\bb{Q}}$. Note that there is a natural identification of $F_\infty \coloneqq F\otimes_\bb{Q}\bb{R}$ with $\bb{R}^{\mrm{I}_F}$.
%
%
Consider the algebraic group 
\[
G\coloneqq\mrm{Res}_{F/\bb{Q}}(\mrm{GL}_{2,F}).
\]
Let $\cl{H}$ denote the Poincar\'e upper half plane. The identity component
$G(\bb{R})^+$ of 
\[
G(\bb{R})= \mrm{GL}_2(F_\infty)\simeq \prod_{\sigma\in I_F}\GL_2(\bb{R}_\sigma)
\]
acts on $\cl{H}^{\mrm{I}_F}$ via M\"obius transformations. We denote by  $i = \sqrt{-1} \in \cl{H}$ the  square root of $-1$ belonging to $\cl{H}$ and we consider the point $\mbf{i} = (i,\dots,i) \in \cl{H}^{\mrm{I}_F}$.


\subsection{Adelic Hilbert cuspforms}

Let $K\le G(\bb{A}_{\bb{Q},f})$ be a compact open subgroup. A \textit{(holomorphic) Hilbert cuspform} of parallel weight $l$ and level $K$ is a function $f:G(\bb{A})\to\bb{C}$ that satisfies the following properties:
\begin{itemize}
\item[$\bullet$] $f(\alpha x u)=f(x)j_{l}(u_\infty,\mathbf{i})^{-1}$ for all $\alpha\in G(\bb{Q})$ and $u\in K\cdot C_{\infty}^+$, where $C_{\infty}^+$ is the stabilizer of $\mathbf{i}$ in $G(\bb{R})^+$ and the automorphy factor is defined by
 
 \[
 j_{l}\big(\gamma,z\big)=\prod_{\sigma\in I_F}(a_\sigma d_\sigma -b_\sigma c_\sigma)^{-1}(c_\sigma z_\sigma +d_\sigma)^l
 \]
 for all
$\gamma=\left(\begin{pmatrix}a_{\sigma}& b_\sigma\\
c_\sigma & d_\sigma \end{pmatrix}\right)_{\sigma\in I_F}
\in G(\bb{R})$ and $(z_\sigma)_{\sigma\in I_F}\in\cl{H}^{\mrm{I}_F}$;
\item[$\bullet$] for every finite adelic point $ x\in G(\bb{A}_{\bb{Q},f})$ the (well-defined) function $f_x:\cl{H}^{\mrm{I}_F}\to\bb{C}$ given by $f_x(z)=f(xu_\infty)j_l(u_\infty,\mathbf{i})$, where $u_\infty\in G(\bb{R})_+$ satisfies $u_\infty\mathbf{i}=z$, is holomorphic;
\item[$\bullet$] for all adelic points $x\in G(\bb{A})$ and for all additive measures on
$F \backslash\bb{A}_F$ we have
 \[\int_{F\backslash\bb{A}_F}f\bigg(\begin{pmatrix}1&a\\0&1\end{pmatrix}x\bigg)da=0;\]
\item[$\bullet$] if $F=\bb{Q}$, then for every finite adelic point $ x\in G(\bb{A}_{\bb{Q},f})$ the function $\vert \mathrm{Im}(z)^{l/2} f_x(z)\vert$ is uniformly bounded on $\cl{H}$.
\end{itemize}

We denote by $S_{l}(K;\bb{C})$ the $\bb{C}$-vector space of Hilbert cuspforms of parallel weight $l$ and level $K$.

\begin{defn}
Let $\mathfrak{N}$ be an integral $\cl{O}_F$-ideal. We define the following compact open subgroups of $G(\bb{A}_{\bb{Q},f})$:
\begin{itemize}
\item $V_0(\mathfrak{N})=\bigg\{\begin{pmatrix}a&b\\c&d\end{pmatrix}\in G(\widehat{\bb{Z}})\;\bigg\lvert \; c\in\mathfrak{N}\widehat{\cal{O}}_F\bigg\}$,
\item $V_1(\mathfrak{N})=\bigg\{\begin{pmatrix}a&b\\c&d\end{pmatrix}\in V_0(\mathfrak{N})\;\bigg\lvert\; d\equiv1 \pmod{\mathfrak{N}\widehat{\cal{O}}_F}\bigg\}$,
\item $V^1(\mathfrak{N})=\bigg\{\begin{pmatrix}a&b\\c&d\end{pmatrix}\in V_0(\mathfrak{N})\;\bigg\lvert\; a\equiv1 \pmod{\mathfrak{N}\widehat{\cal{O}}_F}\bigg\}$,
\item $V(\mathfrak{N})=\bigg\{\begin{pmatrix}a&b\\c&d\end{pmatrix}\in V_0(\mathfrak{N})\;\bigg\lvert\; a,d \equiv1; \, b,c \equiv 0 \pmod{\mathfrak{N}\widehat{\cal{O}}_F}\bigg\}$.
\end{itemize}
\end{defn}
We will frequently write $S_{l}(\frk{N};\bb{C})$ instead of $S_l(V^1(\frk{N});\bb{C})$.

\begin{rk}
    There is a notion of $q$-expansion of Hilbert modular forms (see, for instance, \cite[\S1]{pHida}), and, for any subring $R\subseteq \bb{C}$, we denote by $S_{l}(\frk{N};R)$ the $R$-module of cuspforms in $S_{l}(\frk{N};\bb{C})$ with Fourier coefficients in $R$.
\end{rk}

\subsection{Hecke Theory}

Let $K \le G(\bb{A}_{\bb{Q},f})$ be a compact open subgroup. Suppose that $V(\frk{N})\le K$.

For every $g \in G(\bb{A})$, we define a double coset operator $[KgK]$ acting on the space of adelic Hilbert cuspforms of level $K$. Let
\[
K g K
=
\coprod_i \gamma_i K,
\]
be the decomposition of $KgK$ into left $K$-cosets. Then the action of $[KgK]$ on an adelic Hilbert cuspform $f$ is given by 
\begin{equation*}
\big([KgK]f\big)   (x)
=
\sum_i f(x\gamma_i).
\end{equation*}

\begin{defn}
Let $\mathfrak{q}$ be a prime ideal in $\cl{O}_F$ and let $\varpi_\mathfrak{q}$ be a uniformizer in $\cal{O}_{F,\mathfrak{q}}$. With a slight abuse of notation, we also denote by $\varpi_{\frk{q}}$ the idèle with $\frk{q}$-component equal to $\varpi_\frk{q}$ and all other components equal to $1$. The Hecke operator $T(\varpi_{\frk{q}})$ on $S_{l}(K;\bb{C})$ is defined by
\[
T(\varpi_\mathfrak{q}) =
\Big[ K\begin{pmatrix} 1 & 0 \\ 0 & \varpi_\mathfrak{q} \end{pmatrix} K \Big].
\]
If $V^1(\frk{N})\leq K$ or $\frk{q}\nmid \frk{N}$, the definition does not depend on the choice of uniformizer and we may denote this Hecke operator simply by $T_{\frk{q}}$.
\end{defn}

\begin{defn}
Let $z \in Z_G(\bb{A}_{\bb{Q},f})$ (where $Z_G$ denotes the center of $G$). The diamond operator $\langle z \rangle$ on $S_{l}(K;\bb{C})$ is defined by $(\langle z \rangle f)(x) = f(xz)$.
\end{defn}

Put $\mrm{cl}_F^+(\frk{N})=F^\times\backslash\bb{A}_{F}^\times/F^{\times,+}_\infty (1+\frk{N}\widehat{\cl{O}_F})^\times$. Given a character 
\[
\psi:\mrm{cl}_F^+(\frk{N})\longrightarrow \bb{C}^\times,
\]
we denote by $S_l(\frk{N},\psi,\bb{C})$ the submodule of $S_l(\frk{N},\bb{C})$ on which the diamond operator $\langle z\rangle$ acts by $\psi(z)\vert z\vert_{\bb{A}_{F,f}}^{2-l}$ for all $z\in\bb{A}_{F,f}^\times$. 
Cuspforms in $S_l(\frk{N},\psi,\bb{C})$ are said to have nebentype $\psi$.

\begin{defn}
A cuspform $g\in S_{l}(\frk{N};\bb{C})$ is said to be an \emph{eigenform} if it is an eigenvector for all the Hecke operators of the form $T_{\frk{q}}$, with $\frk{q}$ a finite prime, and for all diamond operators $\langle z\rangle$, with $z\in Z_G(\bb{A}_{\bb{Q},f})$.
\end{defn}

We define the operator $T_p=\prod_{v\mid p} T_v$ acting on $S_{l}(\frk{N};\bb{C})$. If $p\mid \frk{N}$, we may also denote this operator by $U_p$.

\begin{defn}
    An eigenform $g\in S_{l}(\frk{N};\bb{C})$ is said to be \emph{$p$-ordinary} if the eigenvalue of $T_p$ acting on $g$ is a $p$-adic unit (under our fixed embedding $\iota_p:\overline{\bb{Q}}\hookrightarrow\overline{\bb{Q}}_p$).
\end{defn}

\subsection{Galois representations}

Let $g\in S_l(\frk{N},\psi;\bb{C})$ be an eigenform. For a sufficiently large finite extension $L/\bb{Q}_p$, we denote by
\[
\rho_{g}:G_F\longrightarrow \GL(V_g)=\GL_2(L)
\]
the $p$-adic Galois representation attached to $g$ by work of Eichler--Shimura, Deligne, Carayol, Wiles, Taylor. It is unramified outside $p\frk{N}$, and, if $\Fr_\frk{q}$ is an arithmetic Frobenius element at a prime $\frk{q}\nmid p\frk{N}$ and $\varpi_{\frk{q}}$ is a uniformizer of $\cl{O}_{F,\frk{q}}$, then
\[
\det(1-\rho_g(\Fr_\frk{q})X)=1-a_\frk{q}(g)X+\psi(\varpi_\frk{q})^{-1}\vert\frk{q}\vert^{1-l}X^2.
\]

Suppose that $g$ is $p$-ordinary and $l\geq 2$. Let $\frk{p}$ be a prime of $F$ above $p$. Let $\alpha_{\frk{p}}(g)$ denote the unit root of the Hecke polynomial $X^2-a_{\frk{p}}(g)X+\psi(\varpi_\frk{p})^{-1}\vert\frk{p}\vert^{1-l}$. Then,
\[
\rho_g\vert_{G_{F_\frk{p}}} \simeq \begin{pmatrix} \epsilon_1 & \ast \\ 0 & \epsilon_2\end{pmatrix},
\]
where $\epsilon_2$ is the unramified character such that $\epsilon_2(\Fr_{\frk{p}})=\alpha_\frk{p}(g)$.

\subsubsection{Asai representations}
When $F$ is a real quadratic field, we denote by
\[
\mrm{As}(\rho_g)=\otimes\mbox{-}\mathrm{Ind}_{F}^\bb{Q}\rho_g
\]
the $p$-adic Galois representation obtained as the tensor-induction of $\rho_g$ from $F$ to $\bb{Q}$. Suppose that $g$ is $p$-ordinary, and put
\[
\alpha_p(g)=
\begin{cases}
\alpha_{\frk{p}_1}(g)\alpha_{\frk{p}_2}(g)&\textrm{if $(p)=\frk{p}_1\frk{p}_2$ splits in $F$,}\\[0.2em]
\alpha_{\frk{p}}(g)&\textrm{if $(p)=\frk{p}$ is inert in $F$}.
\end{cases}
\]
Then the representation $\mrm{As}(\rho_g)\vert_{G_{\bb{Q}_p}}$ is endowed with a three-step filtration
\[
\mrm{As}(V_g)=\Fil^0 \mrm{As}(V_g) \supset \Fil^1 \mrm{As}(V_g) \supset \Fil^2\mrm{As}(V_g) \supset \Fil^3\mrm{As}(V_g)=0
\]
with graded pieces of dimensions $1$, $2$, and $1$, respectively, and the graded piece $\Gr^0 \mrm{As}(V_g)=\Fil^0\mrm{As}(V_g)/\Fil^1\mrm{As}(V_g)$ is unramified with arithmetic Frobenius acting as multiplication by $\alpha_p(g)$ (see \cite[Cor. 9.2.2]{LLZ-AF}).

\subsection{Hida families}\label{sect Hida families}

Let $K$ be a compact open subgroup of $\GL_2(\widehat{\cl{O}_F})$. Suppose that $V^1(\frk{N})\leq K$ for some integral ideal $\frk{N}$ of $\cl{O}_F$ coprime to $p$. For $\alpha\geq 1$, let $K^1(p^\alpha)=K\cap V^1(p^\alpha)$. Let $\cl{O}$ denote the ring of integers of a sufficiently large finite extension $L/\bb{Q}_p$. The projective limit of $p$-adic Hecke algebras
\[
\frk{h}_F(K;\cl{O})\coloneqq\varprojlim_\alpha h_{l}(K^1(p^\alpha); \cl{O})
\qquad
\text{acts on}
\qquad
\varinjlim_\alpha S_{l}(K^1(p^\alpha); \cl{O})
\]
 through the Hecke operators $T_\frk{q} = \varprojlim_\alpha T_\frk{q}$ and $\langle z\rangle=\varprojlim_\alpha \langle z\rangle$ and is independent of the weight $l$. Since $\mathfrak{h}_F(K; \mathcal O)$ is a compact ring, it can be decomposed as a direct sum of algebras
\[
\mathfrak{h}_F(K; \cl{O}) = \mathfrak{h}^{\ord}_F(K; \cl{O}) \oplus \mathfrak{h}_F^{\text{ss}}(K; \cl{O}),
\]
so that $T_p$ is invertible in $\mathfrak{h}_F^{\ord}(K; \cl{O})$ and topologically nilpotent in $\mathfrak{h}_F^{\text{ss}}(K; \cl{O})$. We denote by
\[
e_{\ord} = \underset{n \to \infty}{\lim} T_p^{n!}
\]
the idempotent corresponding to the ordinary part $\mathfrak{h}^{\ord}_F(K; \cl{O})$.

Let $\Lambda^F=\cl{O}[\![1+p\cl{O}_{F,p}]\!]$. There is a natural embedding $\Lambda^F\longmono \frk{h}_F(K;\cl{O})$ given by $z\mapsto \langle z\rangle$. As shown in \cite[Thm. 2.4]{nearlyHida}, the ordinary Hecke algebra $\mathfrak{h}_F^{\ord}(K; \cl{O})$ is finite and torsion-free over $\Lambda^F$.

\begin{defn}\label{def I-adic cuspforms}
For any $\Lambda^F$-algebra  $\mathbb I$, the space of ordinary $\mathbb I$-adic cuspforms of tame level $K$ is
\[
\bar{\mathbb S}_F^{\ord}(K;\mathbb I) \coloneqq
\mrm{Hom}_{\Lambda^F\mbox{-}\mrm{mod}} \big( \frk{h}_F^{\ord}(K^1(p^\infty); \cl{O}), \mathbb I \big).
\]
When an $\mathbb{I}$-adic cuspform is also a $\Lambda^F$-algebra homomorphism, we call it a \emph{Hida family}.
\end{defn}

Let $\overline{\cl{E}}^+_{\frk{N}p}$ denote the closure in $\cl{O}_{F,p}^\times$ of the set of totally positive units in $\cl{O}_F$ which are congruent to $1$ modulo $\frk{N}p$. Then, as in \cite[eq.~(16)]{FJ}, there is a canonical decomposition
\begin{align*}
    \mrm{cl}_F^+(\frk{N}p^\infty) &\overset{\simeq}\longrightarrow\overline{\cl{E}}^+_{\frk{N}p}\backslash (1+p\cl{O}_{F,p})\times \mrm{cl}_F^+(\frk{N}p) \\
    z & \longmapsto (\xi_z,\theta(z))
\end{align*}
For any finite order character $\chi:\mrm{cl}_F^+(\frk{N}p^\infty)\rightarrow \cl{O}^\times$, we let $\bar{\mathbb S}_F^{\ord}(K,\chi;\bb{I})$ denote the module of ordinary $\bb{I}$-adic cuspforms $\scr{G}$ of level $K$ such that 
\[
\scr{G}(\langle z\rangle)=\chi(z)[\xi_z].
\]
An $\bb{I}$-adic cuspform in $\bar{\mathbb S}_F^{\ord}(K,\chi;\mathbb I)$ will be said to have character $\chi$.

\begin{defn}\label{def:arithmetic-point}
Let $\psi: 1+p\cl{O}_{F,p} \rightarrow \cl{O}^\times$ be a finite order character. For any non-negative integer $l$, the homomorphism $1+p\cl{O}_{F,p}\rightarrow \cl{O}^\times$ defined by $u\mapsto \psi(u)N_{F_p/\bb{Q}_p}(u)^{2-l}$ induces a $\cl{O}$-algebra homomorphism 
\[
\mrm{P}_{l,\psi}:\Lambda^F\longrightarrow \cl{O}.
\]
For a $\Lambda^F$-algebra $\mathbb I$, the set of \emph{arithmetic points}, denoted by $\cal{A}(\mathbb{I})$, is the subset of $\mrm{Hom}_{\cl{O}\mbox{-}\mrm{alg}}(\mathbb{I},\overline{\bb{Q}}_p)$ consisting of homomorphisms that coincide with some $\mrm{P}_{l,\psi}$ when restricted to $\Lambda^F$.
\end{defn}

Note that if we specialize a Hida family $\scr{G}\in \bar{\mathbb S}_F^{\ord}(K,\chi;\bb{I})$ at an arithmetic point $\nu\in\cal{A}(\mathbb{I})$ lying above $\mrm{P}_{l,\psi}$, we obtain an eigenform of parallel weight $l$, level $K$, and nebentype defined by $z\mapsto \chi(z)\psi(\xi_z)N_{F_p/\bb{Q}_p}(\xi_z)^{2-l}\vert z\vert^{l-2}$ for $z\in \bb{A}_{F,f}^\times$.

\subsection{Big Galois representations}

Let $\hg\in \bar{\mathbb S}_F^{\ord}(\frk{N};\bb{I}_{\hg})$ be a Hida family. Let $F_{\hg}$ denote the field of fractions of $\bb{I}_{\hg}$. We denote by
\[
\rho_{\hg}:G_F\longrightarrow \GL(\bb{V}_{\hg})\simeq\GL_2(F_{\hg})
\]
the big Galois representation attached to $\hg$. This representation is unramified outside $p\frk{N}$, and, if $\Fr_\frk{q}$ is an arithmetic Frobenius element at a prime $\frk{q}\nmid p\frk{N}$ and $\varpi_{\frk{q}}$ is a uniformizer of $\cl{O}_{F,\frk{q}}$, then
\[
\det(1-\rho_{\hg}(\Fr_{\frk{q}})X)=1-\hg(T_{\frk{q}})X+\hg(\langle \varpi_{\frk{q}}\rangle^{-1})\vert \frk{q}\vert^{-1} X^2.
\]
Moreover, for any prime $\frk{p}$ of $F$ above $p$,
\[
\rho_{\hg}\vert_{G_{F_\frk{p}}} \simeq \begin{pmatrix} \bm{\epsilon}_1 & \ast \\ 0 & \bm{\epsilon}_2\end{pmatrix},
\]
where $\bm{\epsilon}_2$ is the unramified character defined by  $\bm{\epsilon}_2(\Fr_{\frk{p}})=\hg(T_\frk{p})$. Thus we have a two-step filtration
\[
\bb{V}_{\hg}=\Fil^0 \bb{V}_{\hg} \supset \Fil^1 \bb{V}_{\hg} \supset \Fil^2 \bb{V}_{\hg} =0
\]
of $G_{F_\frk{p}}$-modules with  $\Fil^1\bb{V}_{\hg}$ having rank one, and the graded piece $\Gr^0 \bb{V}_{\hg} =\Fil^0 \bb{V}_{\hg}/\Fil^1 \bb{V}_{\hg}$ is unramified with arithmetic Frobenius acting as multiplication by $\hg(T_\frk{p})$.

For each arithmetic point $\nu:\bb{I}_{\hg}\rightarrow \overline{\bb{Q}}_p$, the representation $\bb{V}_{\hg}\otimes_{\bb{I}_{\hg},\nu} \overline{\bb{Q}}_p$ recovers the representation attached to the $\nu$-specialization $\hg_\nu$ of $\hg$.



\subsubsection{Big Asai representations}
When $F$ is a real quadratic field, we denote by
\[
\mrm{As}(\rho_{\hg})=\otimes\mbox{-}\mathrm{Ind}_{F}^\bb{Q}(\rho_{\hg})
\]
the big Galois representation obtained as the tensor-induction of $\rho_{\hg}$ from $F$ to $\bb{Q}$. The representation $\mrm{As}(\rho_{\hg})\vert_{G_{\bb{Q}_p}}$ is endowed with a three-step filtration
\[
\mrm{As}(\bb{V}_{\hg})=\Fil^0 \mrm{As}(\bb{V}_{\hg}) \supset \Fil^1 \mrm{As}(\bb{V}_{\hg}) \supset \Fil^2\mrm{As}(\bb{V}_{\hg}) \supset \Fil^3\mrm{As}(\bb{V}_{\hg})=0
\]
in which the graded pieces have dimensions $1$, $2$ and $1$, respectively, and the $G_{\bb{Q}_p}$-action on the graded piece $\Gr^0 \mrm{As}(\bb{V}_g)=\Fil^0\mrm{As}(\bb{V}_{\hg})/\Fil^1\mrm{As}(\bb{V}_{\hg})$ is unramified, with arithmetic Frobenius acting as multiplication by $\hg(T_p)$.

\subsection{Selmer conditions}\label{subsec:Selmer}

Let $F$ be a real quadratic field. Let 
\[
\hf\in \bar{\bb{S}}_{\bb{Q}}^{\ord}(N_f,\chi_{\hf};\bb{I}_{\hf}),\quad\hg\in \bar{\bb{S}}_F^{\ord}(\frk{N}_g,\chi_{\hg};\bb{I}_{\hg})
\]
be Hida families passing through the $p$-stabilizations of an elliptic modular form $f\in S_k(N_f,\chi_f;\bb{C})$ and a Hilbert modular form $g\in S_l(\frk{N}_g,\chi_g;\bb{C})$, respectively. Assume that
\[
\chi_f\chi_{g}\vert_{\bb{A}_{\bb{Q},f}^\times}=1.
\]
Let $\mathbb V_{\hf}$ and $\mathbb V_{\hg}$ be the Galois representations attached to $\hf$ and $\hg$, and let $\mrm{As}(\bb{V}_{\hg})$ be the Asai representation introduced above. We decompose the cyclotomic character $\epsilon_{\cyc}:G_\bb{Q}\rightarrow \bb{Z}_p^\times=\mu_{p-1}\times (1+p\bb{Z}_p)$ as 
\[
\epsilon_{\cyc}=\omega_{\cyc}\theta_{\cyc},
\]
with $\omega_{\cyc}$ and $\theta_{\cyc}$ taking values in $\mu_{p-1}$ and $1+p\bb{Z}_p$, respectively. Let 
\[
\bm{\kappa}^{1/2}:G_{\bb{Q}}\rightarrow (\Lambda\otimes \Lambda^F)^\times
\]
denote the character defined by $\sigma\mapsto [\theta_{\cyc}(\sigma)^{1/2}]\otimes [\theta_{\cyc}(\sigma)^{1/2}]$, and put 
\[
\bm{\eta}=\omega_{\cyc}^{3-k/2-l}:G_{\bb{Q}}\longrightarrow \mu_{p-1}\subset \bb{Z}_p^\times.
\]

\begin{defn}
Let $\mathbb V_{\hf \hg}^{\dag}$ by the twist of $\mathbb V_{\hf} \hat \otimes \As(\mathbb V_{\hg})$ given by
\[
\mathbb V_{\hf \hg}^{\dag} \coloneqq \mathbb V_{\hf} \hat \otimes \As(\mathbb V_{\hg})(\Xi_{\hf\hg}),
\]
where  $\Xi_{\hf\hg} = \epsilon_{\cyc}^{-1} \bm{\eta}\bm{\kappa}^{-1/2}:G_{\bb{Q}}\rightarrow (\Lambda\otimes \Lambda^F)^\times$. 
\end{defn}
%
%
Arising from the filtration on $\mathbb V_{\hf}$ and $\As(\mathbb V_{\hg})$ discussed above, the Galois representation 
$\mathbb V_{\hf\hg}^{\dag}$ is naturally endowed with a $G_{\bb{Q}_p}$-stable filtration
\[ \mathbb V_{\hf\hg}^{\dag} = \mathscr{F}^0\mathbb V_{\hf\hg}^\dag \supset \mathscr{F}^1\mathbb V_{\hf\hg}^\dag \supset \mathscr{F}^2\mathbb V_{\hf\hg}^\dag \supset \mathscr{F}^3\mathbb V_{\hf\hg}^\dag \supset \mathscr{F}^4\mathbb V_{\hf\hg}^\dag = 0,
\]
where the submodules have rank 8, 7, 4, 1 and 0, respectively. In particular, the piece of most interest to us in this paper is
\[ 
\mathscr{F}^2\mathbb V_{\hf\hg}^\dag :=\bigl(\Fil^1 \bb{V}_{\hf} \hat \otimes \Fil^1 \mrm{As}(\bb{V}_{\hg}) + \Fil^0 \bb{V}_{\hf} \hat \otimes \Fil^2 \mrm{As}(\bb{V}_{\hg})\bigr)(\Xi_{\hf\hg}).
\]
We also define the $G_{\bb{Q}_p}$-stable rank-four submodule
\[
\mathbb V_{\hf\hg}^f:=\bigl(\Fil^1 \bb{V}_{\hf} \hat\otimes_{\mathcal{O}}\As(\mathbb{V}_{\hg})\bigr)(\Xi_{\hf\hg}).
\]

Specializing the Hida family $\hg$ to $g$, we obtain a rank-eight $\bb{I}_{\hf}$-module
\[
\mathbb V_{\hf g}^{\dag} \coloneqq \mathbb V_{\hf} \hat \otimes_{\mathcal{O}} \As(\mathbb V_{g})(\Xi_{\hf g})
\]
and $G_{\bb{Q}_p}$-stable rank-four $\bb{I}_{\hf}$-submodules $\mathscr{F}^2\mathbb V_{\hf g}$ and $\mathbb V_{\hf g}^f$.

Fix a finite set $\Sigma$ of places of $\mathbb Q$ containing $\infty$ and the primes dividing $N_fN_{F/\mathbb{Q}}(\mathfrak{N}_g)p$, and let $\mathbb Q^\Sigma$ be the maximal extension of $\mathbb Q$ unramified outside $\Sigma$.

\begin{defn}\label{def:Sel}
For $\mathcal{L}\in\{\bal,\mathcal{F}\}$ define the Selmer group $\Sel_{\mathcal{L}}(\mathbb V_{\hf \hg}^{\dag})$ by
\[
\Sel_{\mathcal{L}}(\mathbb V_{\hf \hg}^{\dag}) = \ker \bigg(
H^1(\mathbb Q^\Sigma/\mathbb Q, \mathbb V_{\hf \hg}^{\dag}) \longrightarrow
\prod_{v \in \Sigma\smallsetminus\{p,\infty\}} H^1(\mathbb Q_v^{\nr},\mathbb V_{\hf \hg}^{\dag})\times
\frac{H^1(\mathbb Q_p,\mathbb V_{\hf \hg}^\dag)}{H^1_{\mathcal{L}}(\mathbb Q_p,\mathbb V_{\hf \hg}^\dag)}\bigg),
\]
where
$\mathbb Q_v^{\nr}$ denotes the maximal unramified extension of $\mathbb Q_v$ and
\[
H_{\mathcal L}^1(\mathbb Q_p, \mathbb V_{\hf \hg}^{\dag}) =
\begin{cases}
\ker \bigl(H^1(\mathbb Q_p, \mathbb V_{\hf \hg}^{\dag}) \longrightarrow H^1(\mathbb Q_p, \mathbb V_{\hf \hg}^{\dag}/\mathscr{F}^2\mathbb V_{\hf \hg}^{\dag}) \bigr) &\textrm{if $\mathcal{L}={\bal}$},\\[0.5em]
\ker \bigl(H^1(\mathbb Q_p, \mathbb V_{\hf \hg}^{\dag}) \longrightarrow H^1(\mathbb Q_p^{}, \mathbb V_{\hf \hg}^{\dag}/ \mathbb V_{\hf \hg}^{f}) \bigr) &\textrm{if $\mathcal{L}=\mathcal{F}$}.
\end{cases}
\]
We call $\Sel_{\bal}(\mathbb V_{\hf \hg}^{\dag})$ (resp. $\Sel_{\mathcal F}(\mathbb V_{\hf \hg}^{\dag})$)  the \emph{balanced} (resp. \emph{$\hf$-unbalanced}) Selmer group. 

Considering instead the $g$-specialized modules, we similarly define the Selmer groups corresponding to the previous local conditions $\Sel_{\bal}(\mathbb V_{\hf g}^{\dag})$  and $\Sel_{\mathcal F}(\mathbb V_{\hf g}^{\dag})$.
\end{defn}

\section{Geometry of modular curves and Hilbert modular varieties}

In this short section we review the definitions of certain maps and correspondences in the elliptic and Hilbert settings for our later use in the paper. 

\subsection{Degeneracy maps}

Let $K$ be a compact open subgroup of $\mathrm{GL}_2(\widehat{\cl{O}}_F)$ containing $V^1(\frk{N})$ for some ideal $\frk{N}$ of $\cl{O}_F$. Let $\frk{q}$ be a prime ideal of $\cl{O}_{F}$, and let $\varpi_\frk{q}$ be a uniformizer in $\cl{O}_{F,\frk{q}}$. Suppose that $(\frk{q},\frk{N})=1$. Given a non-negative integer $\alpha$ and a positive integer $r$, we put 
\[
K_0^1(\frk{q}^{\alpha},\frk{q}^{\alpha+r})= K\cap V^1(\frk{q}^\alpha)\cap V_0(\frk{q}^{\alpha+r}).
\]
Consider the modular variety $S_{10}(\frk{q}^\alpha,\frk{q}^{\alpha+r})=\mathrm{Sh}(G,K^1_0(\frk{q}^\alpha,\frk{q}^{\alpha+r}))$. Let $\eta_\frk{q}\in \GL_2(\bb{A}_{F,f})$ be the element defined by
\[
\eta_{\frk{q}\frk{q}}=\begin{pmatrix}
    1 & 0 \\ 0 & \varpi_\frk{q}
\end{pmatrix},\quad
\eta_{\frk{q}v}=\begin{pmatrix}
    1 & 0 \\ 0 & 1
\end{pmatrix} \quad\text{for $v\neq \frk{q}$.}
\]
We denote by $\mu_\frk{q}$ the degeneracy map $S_{10}(\frk{q}^{\alpha+1},\frk{q}^{\alpha+r})\rightarrow S_{10}(\frk{q}^\alpha,\frk{q}^{\alpha+r})$ given by $[x,g] \mapsto [x,g]$. We also have the two degeneracy maps 
\[
\pi_{1,\frk{q}},\pi_{2,\frk{q}}:S_{10}(\frk{q}^\alpha,\frk{q}^{\alpha+r})\longrightarrow S_{10}(\frk{q}^\alpha,\frk{q}^{\alpha+r-1})
\]
defined by $[x,g]\mapsto [x,g]$ and $[x,g]\mapsto [x,g\eta_\frk{q}]$, respectively.

\subsection{Hecke correspondences}

Let $K_1$ and $K_2$ be compact open subgroups of $\mathrm{GL}_2(\bb{A}_{F,f})$ and let $g\in \GL_2(\bb{A}_{F,f})$. Then we have a correspondence
\[
\begin{tikzcd}[column sep=large]
 & \mathrm{Sh}(G,gK_1g^{-1}\cap K_2) \arrow[ld] \arrow[rd] & \\ \mathrm{Sh}(G,K_1) & & \mathrm{Sh}(G,K_2)
\end{tikzcd}
\]
which induces contravariant maps
\[
[K_1gK_2]:H^\ast(\mathrm{Sh}(G,K_2),\bb{Z}_p)\longrightarrow H^\ast(\mathrm{Sh}(G,K_1),\bb{Z}_p)
\]
and covariant maps
\[
[K_1gK_2]':H^\ast(\mathrm{Sh}(G,K_1),\bb{Z}_p)\longrightarrow H^\ast(\mathrm{Sh}(G,K_2),\bb{Z}_p).
\]

Let $K$, $\frk{q}$ and $\eta_{\frk{q}}$ be as in the previous subsection. Then we define the Hecke operators
\begin{gather*}
    T_\frk{q}=[K^1(\frk{q}^\alpha)\eta_{\frk{q}}K^1(\frk{q}^\alpha)]=\pi_{2,\frk{q},\ast}\circ \pi_{1,\frk{q}}^\ast, \\
    T_\frk{q}'=[K^1(\frk{q}^\alpha)\eta_{\frk{q}}K^1(\frk{q}^\alpha)]'=\pi_{1,\frk{q},\ast}\circ\pi_{2,\frk{q}}^\ast.
\end{gather*}
When $\alpha>0$, we may also denote these operators by $U_{\frk{q}}$ and $U_{\frk{q}}'$, respectively.

Let $z\in \bb{A}_{F,f}^\times$. Then, we can define a homomorphism
\begin{align*}
    \frk{T}_{z^{-1}} : \mathrm{Sh}(G,K) &\longrightarrow \mathrm{Sh}(G,K) \\
    [x,g] &\longmapsto \left[x, g\begin{pmatrix} z & 0 \\ 0 & z \end{pmatrix}\right].
\end{align*}
The corresponding contravariant (resp. covariant) operator on $H^\ast(\mathrm{Sh}(G,K),\bb{Z}_p)$ will be denoted by $\langle z\rangle$ (resp. $\langle z\rangle'$). Note that $\langle z\rangle'=\langle z^{-1}\rangle$.

\subsection{Atkin--Lehner maps}\label{subsec:AL}

Let $K=V^1(N)$ for some integer $N$. Let $\iota: \GL_2(\bb{A}_{F,f})\rightarrow \GL_2(\bb{A}_{F,f})$ denote the involution defined by $g\mapsto \det(g)^{-1}g$. Let $\tau_N=\begin{pmatrix} 0 & -1 \\ N & 0 \end{pmatrix}$ and let $\frk{T}_{\tau_N}: \mathrm{Sh}(G,K)\rightarrow \mathrm{Sh}(G,\tau_N K \tau_N^{-1})$ be the homomorphism defined by $[x,g]\mapsto [x,g\tau_N^{-1}]$. Also note that $\iota$ induces a map, that we will denote in the same way, $\iota:\mathrm{Sh}(G,\tau_N K\tau_N^{-1})\rightarrow \mathrm{Sh}(G,K)$ defined by $[x,g]\mapsto [x,\iota(g)]$. We define the homomorphism 
\[
\lambda_N:\mathrm{Sh}(G,K)\longrightarrow \mathrm{Sh}(G,K)
\]
as the composition $\lambda_N=\iota\circ \frk{T}_{\tau_N}$. 

Let $\sigma\in G_{\mathbb{Q}}$ and take $s\in\bb{A}_{\bb{Q},f}^\times$ such that $\mrm{art}_{\bb{Q}}(s)=\sigma_{\vert\bb{Q}^{\ab}}$, where $\mrm{art}_{\bb{Q}}$ denotes the geometrically normalized Artin map. It follows from \cite[Lemma~7.10]{FJ} that
\[
\frk{T}_{s^{-1}}\circ \lambda_N\circ \sigma=\sigma\circ \lambda_N.
\]
In particular, $\lambda_N$ is defined over $\bb{Q}(\zeta_N)$. Note that the above relation implies
\[
\sigma \circ\lambda_{N\ast}\circ \sigma^{-1}=\langle s^{-1}\rangle \circ \lambda_{N\ast}.
\]

\section{Big Hirzebruch--Zagier classes}
\label{sec:cycles}

\subsection{A compatible collection of cycles}\label{subsec:cycles}

In this section we construct big Hirzebruch--Zagier cycles following the approach of \cite{DR3}. The construction works both when $p$ splits in $F$ and when $p$ is inert in $F$.


Let $F$ be a real quadratic number field, let $\cl{O}_F$ be its ring of integers and let $\sigma\in\Gal(F/\bb{Q})$ be the nontrivial automorphism of $F$. We denote by $\cl{O}_{F,p}$ the $\bb{Z}_p$-algebra $\cl{O}_F\otimes\bb{Z}_p$.

Given a $\bb{Z}_p$-module (respectively $\cl{O}_{F,p}$-module) $\Omega$, we denote by $\Omega'$ the set of primitive elements in $\Omega$, i.e., the set of elements which are not divisible by any non-unit element of $\bb{Z}_p$ (respectively $\cl{O}_{F,p}$).

For each integer $\alpha\geq 1$, define the set
\[
\Sigma_\alpha=\left(\bb{Z}/p^\alpha\bb{Z}\times \bb{Z}/p^\alpha\bb{Z}\right)'\times\left(\cl{O}_F/p^\alpha \times \cl{O}_F/p^\alpha\right)'
\]
equipped with the diagonal action of $\mathrm{GL}_2(\bb{Z}/p^\alpha\bb{Z})$ by multiplication on the right. The quotient $\Sigma_\alpha/\mathrm{SL}_2(\bb{Z}/p^\alpha\bb{Z})$ is equipped with a determinant map
\[
D : \Sigma_\alpha/\mathrm{SL}_2(\bb{Z}/p^\alpha\bb{Z}) \longrightarrow \cl{O}_F/p^\alpha\cl{O}_F\times \left(\cl{O}_F/p^\alpha\cl{O}_F\right)^{\mathrm{tr=0}}
\]
defined by
\[
D\left((x_0,y_0),(x_1,y_1)\right)=\left(\begin{vmatrix} x_0 & y_0 \\ x_1 & y_1 \end{vmatrix}, \begin{vmatrix} x_1 & y_1 \\ x_1^{\sigma} & y_1^{\sigma} \end{vmatrix}\right)
\]
For each pair $(t_0,t_1)\in \cl{O}_F/p^\alpha\cl{O}_F\times (\cl{O}_F/p^\alpha\cl{O}_F)^{\mathrm{tr=0}}$, we define
\[
\Sigma_{\alpha}[t_0,t_1]=\left\lbrace (v_0,v_1)\in \Sigma_\alpha : D(v_0,v_1)=(t_0,t_1)\right\rbrace.
\]

\begin{lemma}
The group $\mathrm{SL}_2(\bb{Z}/p^\alpha\bb{Z})$ acts simply transitively on $\Sigma_\alpha[t_0,t_1]$ if
\[
[t_0,t_1]\in I_\alpha \coloneqq \left(\cl{O}_F/p^\alpha\cl{O}_F\right)^\times\times \left(\cl{O}_F/p^\alpha\cl{O}_F\right)^{\times,\mathrm{tr=0}}.
\]
\end{lemma}
\begin{proof}
Let $(v_0,v_1),(w_0,w_1)\in\Sigma_\alpha$ satisfy $D(v_0,v_1)=D(w_0,w_1)$ in $I_\alpha$. Since $\det(v_1, v_1^\sigma)=\det(w_1, w_1^\sigma)$ is invertible, there is $\gamma\in\mathrm{SL}_2(\mathcal{O}_F/p^\alpha)$ such that
\[
\begin{pmatrix}
v_1\\
v_1^\sigma
\end{pmatrix}\gamma=\begin{pmatrix}
w_1\\
w_1^\sigma
\end{pmatrix}.
\]
We deduce that $\gamma^\sigma=\gamma$, i.e. $\gamma\in\mathrm{SL}_2(\bb{Z}/p^\alpha\bb{Z})$. We are left to show that $v_0\cdot\gamma=w_0$. There are $a,b\in \mathcal{O}_F/p^\alpha$ such that
\[
v_0=a\cdot v_1+a^\sigma\cdot v_1^\sigma,\qquad w_0=b\cdot w_1+b^\sigma\cdot w_1^\sigma,
\]
and we can compute that
\[
a^\sigma\cdot\det(v_1^\sigma, v_1)=\det(v_0, v_1)=\det(w_0,w_1)=b^\sigma\cdot\det(w_1^\sigma, w_1),
\]
thus $a^\sigma=b^\sigma$. The claim follows.
\end{proof}

Let $K$ be a compact open subgroup of $\mathrm{GL}_2(\widehat{\cl{O}}_F)$. We assume that $V^1(\frk{N})\subseteq K$ for some ideal $\frk{N}$ of $\cl{O}_F$ with $(p,\frk{N})=1$. Let $U=K\cap \mathrm{GL}_2(\mathbb{A}_{\mathbb{Q},f})$. For each positive integer $\alpha$, let 
\begin{align*}
K^1(p^\alpha)=K\cap V^1(p^\alpha),&\quad K^1_0(p^{\alpha+1})=K\cap V_0(p^{\alpha+1})\cap V^1(p^\alpha),\\ 
U^1(p^\alpha)=U\cap V^1(p^\alpha),&\quad U^1_0(p^\alpha)=U\cap V_0(p^{\alpha+1})\cap V^1(p^\alpha)
\end{align*}
and $U(p^\alpha)=U\cap V(p^\alpha)$. 


For each positive integer $\alpha$, consider the Hilbert modular surface $S_1(p^\alpha)= \mathrm{Sh}(G,K^1(p^\alpha))$, which is defined over $\bb{Q}$. The complex points of $S_1(p^\alpha)$ are given by 
\[
S_1(p^\alpha)(\bb{C})= \mathrm{GL}_2(F)^+\backslash \cl{H}^2\times \mathrm{GL}_2(\bb{A}_{F,f})/K^1(p^\alpha).
\]
We will also need to consider the Hilbert modular surface $S_{10}(p^{\alpha+1})=\mathrm{Sh}(G,K_{0}^1(p^{\alpha+1}))$.

For each positive integer $\alpha$ we consider the modular curves 
\[
Y_1(p^\alpha)= \mathrm{Sh}(\GL_2,U^1(p^\alpha)),\quad 
Y(p^\alpha)= \mathrm{Sh}(\GL_2,U(p^\alpha)),
\]
which are defined over $\bb{Q}$. We fix a connected component $\bb{Y}(p^\alpha)$ of $Y(p^\alpha)_{\bb{Q}(\zeta_{p^\alpha})}$
in such a way that points in $\bb{Y}(p^{\alpha+1})$ map to points in $\bb{Y}(p^\alpha)$ under the natural map from $Y(p^{\alpha+1})$ to $Y(p^\alpha)$. For instance, using the complex uniformization
\[
Y(p^\alpha)(\bb{C})= \mathrm{GL}_2(\bb{Q})^+\backslash \cl{H}\times \mathrm{GL}_2(\bb{A}_{\bb{Q},f})/U(p^\alpha),
\]
we could choose $\bb{Y}(p^\alpha)$ to be the image of the embedding $\Gamma_U(p^\alpha)\backslash\cl{H}\hookrightarrow Y(p^\alpha)(\bb{C})$ taking $z$ to $[z,1]$, where $\Gamma_U(p^\alpha)=\mathrm{GL}_2(\bb{Q})^+\cap U(p^\alpha)$.
We will also need to consider the modular curve $Y_{10}(p^{\alpha+1})=\mathrm{Sh}(\GL_2,U_{0}^1(p^{\alpha+1}))$.

Let $(v_0,v_1)=((c_0,d_0),(c_1,d_1))$ be an element in $\Sigma_\alpha$ and choose matrices $\delta_0\in\mathrm{SL}_2(\bb{Z}_p)$ and $\delta_1\in \mathrm{SL}_2(\cl{O}_{F,p})$ such that the reduction modulo $p^\alpha$ of their bottom row is given by $v_0$ and $v_1$, respectively.

Let $a\in \mathrm{GL}_2(\widehat{\bb{Z}})$ be the element with local components $a_v=\mathbbm{1}_2$ for $v\neq p$ and $a_p=\delta_0^{-1}$. Let $b\in \mathrm{GL}_2(\widehat{\cl{O}}_F)$ be the element with local components $a_v=\mathbbm{1}_v$ for $v\nmid p$ and $b_p=\delta_1^{-1}$. Then we can define a morphism of varieties defined over $\bb{Q}$
\[
\varphi_{(v_0,v_1)} : Y(p^\alpha)\longrightarrow Y_1(p^\alpha)\times S_1(p^\alpha)
\]
by the rule
\[
[x,g] \mapsto \left([x,ga], [x,x,gb]\right),
\]
where we use the same notation for an element $g\in \mathrm{GL}_2(\bb{A}_{\bb{Q},f})$ and its natural inclusion in $\mathrm{GL}_2(\bb{A}_{F,f})$.
We denote by $\Delta_\alpha(v_0,v_1)=\varphi_{(v_0,v_1)}(\bb{Y}(p^\alpha))$ the scheme-theoretic image of $\bb{Y}(p^\alpha)$ by $\varphi_{(v_0,v_1)}$. Since $\bb{Y}(p^\alpha)$ is only defined over $\bb{Q}(\zeta_{p^\alpha})$, so is $\Delta_\alpha(v_0,v_1)$.

It is clear that the cycle $\Delta_\alpha(v_0,v_1)\subset Y_1(p^\alpha)\times S_1(p^\alpha)$ depends only on the class of $(v_0,v_1)$ in $\Sigma_\alpha/\mathrm{SL}_2(\bb{Z}/p^\alpha\bb{Z})$. In particular, if $(t_0,t_1)\in I_\alpha$, we can define
\[
\Delta_\alpha[t_0,t_1]\in \ch^2(Y_1(p^\alpha)\times S_1(p^\alpha))(\bb{Q}(\zeta_\alpha))
\]
to be equal to the cycle $\Delta_\alpha(v_0,v_1)$ for any $(v_0,v_1)\in \Sigma_\alpha[t_0,t_1]$.

For $m\in (\bb{Z}/p^\alpha\bb{Z})^\times$, we denote by $\sigma_m$ the element in $\Gal(\bb{Q}(\zeta_{p^\alpha})/\bb{Q})$ defined by $\zeta\mapsto \zeta^m$.

\begin{lemma}\label{lemma:cycles-galoisaction}
Let $[t_0,t_1]\in I_\alpha$. Then, for all $\sigma_m\in\Gal(\bb{Q}(\zeta_{p^\alpha})/\bb{Q})$,
\[
\sigma_m \Delta_\alpha[t_0,t_1]=\Delta_\alpha[m \cdot t_0,\ m \cdot t_1].
\]
Moreover, for all $z_0\in \bb{Z}_p^\times$, $z_1\in \cl{O}_{F,p}^\times$
\[
\langle z_0^{-1},z_1^{-1}\rangle \Delta_\alpha[t_0,t_1]=\Delta_\alpha[(z_0z_1) \cdot t_0,\ \mathrm{N}_{F/\bb{Q}}(z_1) \cdot t_1].
\]
\end{lemma}
\begin{proof}

To prove the first statement, we first note that the map
\begin{align*}
Y(p^\alpha)=\mathrm{GL}_2(\bb{Q})^+\backslash \cl{H}\times \mathrm{GL}_2(\bb{A}_{\bb{Q},f})/U(p^\alpha) &\longrightarrow \bb{Q}^{\times,+}\backslash \bb{A}_{\bb{Q},f}^\times/\det(U(p^\alpha)) \\
[x,g] &\longmapsto \det(g)
\end{align*}
yields an isomorphism
\[
\pi_0(Y(p^\alpha))\simeq \bb{Q}^{\times,+}\backslash \bb{A}_{\bb{Q},f}^\times/\det(U(p^\alpha)),
\]
Moreover, for all $s\in \bb{A}_{\bb{Q},f}^\times$, the action of $\sigma=\mathrm{art}_{\bb{Q}}(s)\in \Gal(\bb{Q}^\mathrm{ab}/\bb{Q})$ on $\pi_0(Y(p^\alpha))$ is described, via the above isomorphism, by $t\mapsto st$.

Let $m\in\bb{Z}_p^\times$. Let $h\in\mathrm{GL}_2(\bb{A}_{\bb{Q},f})$ be the element with local components $h_p=\mathrm{diag}(1,m)$ and $h_v=\mathbbm{1}_2$ for $v\neq p$. Then $\mathrm{art}_\bb{Q}(\det(h))=\sigma_m$ in $\Gal(\bb{Q}(\zeta_{p^\alpha})/\bb{Q})$. Therefore, the $\bb{Q}$-rational morphism
\begin{align*}
    \varphi_m\colon Y(p^\alpha) &\longrightarrow Y(p^\alpha) \\
    [x,g] &\longmapsto [x,gh]
\end{align*}
descends to $\sigma_m$ on $\pi_0(Y(p^\alpha))$. Let $(v_0,v_1)\in\Sigma_\alpha[t_0,t_1]$.
Since $\varphi_{(v_0,v_1)}\colon Y(p^\alpha)\to Y_1(p^\alpha)\times S_1(p^\alpha)$ is also a $\bb{Q}$-rational morphism,
\[
\sigma_m \Delta_\alpha[t_0,t_1] = \sigma_m\left(\varphi_{(v_0,v_1)}(\bb{Y}(p^\alpha))\right)=\varphi_{(v_0,v_1)}\left(\sigma_m(\bb{Y}(p^\alpha))\right)=\varphi_{(v_0,v_1)}\left(\varphi_m(\bb{Y}(p^\alpha))\right).
\]

For $j=0,1$, let $v_j=(c_j,d_j)$ and
\[
\delta_j=\begin{pmatrix} a_j & b_j \\ c_j & d_j \end{pmatrix}.
\]
Then, we have
\begin{gather*}
h_p \delta_0^{-1}U^{1}(p^\alpha)_p=\begin{pmatrix}
d_0 & -m^{-1}b_0 \\ -mc_0 & a_0
\end{pmatrix}U^{1}(p^\alpha)_p, \\
h_p \delta_1^{-1}K^{1}(p^\alpha)_p=\begin{pmatrix}
d_1 & -m^{-1}b_1 \\ -mc_1 & a_1
\end{pmatrix}K^{1}(p^\alpha_p).
\end{gather*}
This shows that $\varphi_{(v_0,v_1)}\circ\varphi_m=\varphi_{(v_0',v_1')}$, where $v_j'=(mc_j,d_j)$ for $j=0,1$. Since $D(v_0',v_1')=(m\cdot t_0,m\cdot t_1)$, we deduce that
\[
\sigma_m\Delta_\alpha[t_0,t_1]=\varphi_{(v_0',v_1')}(\bb{Y}(p^\alpha))=\Delta_\alpha[m\cdot t_0,m\cdot t_1].
\]

The result for diamond operators follows after noting that
\begin{gather*}
\delta_0^{-1}\begin{pmatrix}
z_0 & 0\\0 & z_0
\end{pmatrix}U^{1}(p^\alpha)_p=\begin{pmatrix}
z_0 d_0&-z_0^{-1} b_0 \\ -z_0 c_0 & z_0^{-1} a_0
\end{pmatrix}U^{1}(p^\alpha)_p, \\
\delta_1^{-1}\begin{pmatrix}
z_1 & 0 \\ 0 & z_1
\end{pmatrix}K^{1}(p^\alpha)_p=\begin{pmatrix}
z_1 d_1 & -z_1^{-1} b_1 \\ -z_1 c_1 & z_1^{-1} a_1
\end{pmatrix}K^{1}(p^\alpha)_p.
\end{gather*}
\end{proof}

\begin{lemma}\label{lemma:cycle-relations}
Let $\alpha\geq 1$ and let $[t_0',t_1']\in I_{\alpha+1}$ map to $[t_0,t_1]\in I_\alpha$. Then
\begin{align*}
(\varpi_1,\varpi_1)_{\ast}\Delta_{\alpha+1}[t_0',t_1'] &=p^3\Delta_{\alpha}[t_0,t_1], \\
(\varpi_2,\varpi_2)_\ast \Delta_{\alpha+1}[t_0',t_1'] &=(U_p,U_p)\Delta_{\alpha}[t_0,t_1].
\end{align*}
Moreover, the cycles $\Delta_\alpha[t_0,t_1]$ satisfy the distribution relation
\[
\sum_{[t_0',t_1']}\Delta_{\alpha+1}[t_0',t_1']=(\varpi_1,\varpi_1)^\ast\Delta_\alpha[t_0,t_1],
\]
where the sum is taken over pairs $[t_0',t_1']\in I_{\alpha+1}$ which map to $[t_0,t_1]\in I_\alpha$.
\label{lemma:cycles-relations}
\end{lemma}
\begin{proof}
Choose $(v_0',v_1')\in\Sigma_{\alpha+1}[t_0',t_1']$ and let $(v_0,v_1)$ be its image in $\Sigma_\alpha[t_0,t_1]$.

We have a commutative diagram
\[
\begin{tikzcd}[column sep=large]
\bb{Y}(p^{\alpha+1}) \arrow[d] \arrow[r, hook, "\varphi_{(v_0',v_1')}"] & Y_1(p^{\alpha+1})\times S_1(p^{\alpha+1}) \arrow[d] \\
\bb{Y}(p^\alpha) \arrow[r, hook, "\varphi_{(v_0,v_1)}"] & Y_1(p^\alpha)\times S_1(p^\alpha)
\end{tikzcd}
\]
in which the horizontal maps are closed embeddings. Since the left vertical map is finite of degree $p^3$, we deduce the first identity in the lemma.

To prove the second identity, we observe that in the diagram
\[
\begin{tikzcd}[column sep=large]
\bb{Y}(p^{\alpha+1}) \arrow[r, hook, "\varphi_{(v_0',v_1')}"] \arrow[rd, hook, "\phi_{(v_0',v_1')}"'] & Y_1(p^{\alpha+1})\times S_1(p^{\alpha+1}) \arrow[d, "{(\mu,\mu)}"] \\
& Y_{01}(p^{\alpha+1})\times S_{01}(p^{\alpha+1})
\end{tikzcd}
\]
the diagonal arrow is also a closed embedding. The map $\phi_{(v_0',v_1')}$ above is simply defined as the composition of the other two maps in the diagram, and we define $\Delta_{\alpha+1}^\flat[t_0,t_1]$ as the scheme-theoretic image of $\phi_{(v_0',v_1')}$. The notation here reflects the fact that $\Delta_{\alpha+1}^\flat[t_0,t_1]$ only depends on $[t_0,t_1]\in I_{\alpha}$, and it is clear from the definitions that $(\mu,\mu)_\ast\Delta_{\alpha+1}[t_0',t_1']=\Delta_{\alpha+1}^\flat[t_0,t_1]$. Now, in the diagram
\[
\begin{tikzcd}[column sep=large]
\bb{Y}(p^{\alpha+1}) \arrow[r, hook, "\phi_{(v_0',v_1')}"] \arrow[d] & Y_{01}(p^{\alpha+1})\times S_{01}(p^{\alpha+1}) \arrow[d, "{(\pi_1,\pi_1)}"] \\
\bb{Y}(p^\alpha) \arrow[r, hook, "\varphi_{(v_0,v_1)}"] & Y_{1}(p^{\alpha})\times S_{1}(p^{\alpha})
\end{tikzcd}
\]
the horizontal arrows are again closed embeddings, whereas the vertical maps are both finite of degree $p^3$. Hence, we obtain the identity
\[(\pi_1,\pi_1)^\ast\Delta_{\alpha}[t_0,t_1]=\Delta_{\alpha+1}^\flat[t_0,t_1].
\]
Altogether, we have
\[
(\varpi_2,\varpi_2)_\ast \Delta_{\alpha+1}[t_0',t_1']=(\pi_2,\pi_2)_\ast (\pi_1,\pi_1)^\ast \Delta_{\alpha}[t_0,t_1]=(U_p,U_p)\Delta_{\alpha}[t_0,t_1].
\]

Finally, to prove the last statement, observe that the map $(\mu,\mu)$ is a Galois cover and the sum in the statement is taken over the Galois-translates of a fixed $\Delta_{\alpha+1}[t_0',t_1']$. Therefore,
\[
\sum_{[t_0',t_1']}\Delta_{\alpha+1}[t_0',t_1']=(\mu,\mu)^\ast (\mu,\mu)_\ast \Delta_{\alpha+1}[t_0',t_1'] = (\varpi_1,\varpi_1)^\ast \Delta_{\alpha}[t_0,t_1].
\]
\end{proof}

\subsection{Galois cohomology classes}\label{subsec:classes}

Assume from now on that $V^1(N)\leq K$ for some integer $N$ such that $p\nmid N$. We shorten notation by writing 
\[
Z_\alpha=Y_1(p^\alpha)\times S_1(p^\alpha).
\]
Let $E$ be a finite extension of $\bb{Q}_p$ and let $\cl{O}$ denote its ring of integers.



We want to use the cycles $\Delta_{\alpha}[t_0,t_1]\subseteq Z_\alpha$ to define compatible families of Galois cohomology classes with coefficients in $H^3_{\et}(Z_{\alpha,\overline{\bb{Q}}},\cl{O}(2))$. To that end, we first modify the cycles $\Delta_{\alpha}[t_0,t_1]$ to make them null-homologous, i.e., to obtain cycles in the kernel of the \'etale cycle class map
\[
\ch^2(Z_\alpha)(\bb{Q}(\zeta_\alpha))\longrightarrow H_{\et}^4(Z_{\alpha,\overline{\bb{Q}}},\bb{Z}_p(2)).
\]
Let $\ell$ be a rational prime and define
\[
\Delta_\alpha^\circ[t_0,t_1]=(\ell+1-(T_\ell,\id))\Delta_\alpha[t_0,t_1]\in \ch^2(Z_\alpha)(\bb{Q}(\zeta_\alpha)).
\]
Since, as in the proof of \cite[Prop.~5.10]{FJ}, the correspondence $\ell+1-(T_\ell,\id)$ annihilates the cohomology group $H_{\et}^4(Z_{\alpha,\overline{\bb{Q}}},\bb{Z}_p(2))$, the cycles $\Delta_\alpha^\circ[t_0,t_1]$ are null-homologous.

Fix an element $[a,b]\in I_1$. We denote by $\vert\Delta_1^\circ[a,b]\vert$ the support of the cycle $\Delta_1^\circ[a,b]$ and we define $\Delta_{\alpha}^\circ[[a,b]]$ to be the variety fitting in the Cartesian diagram
\[
\begin{tikzcd}
\Delta_\alpha^\circ[[a,b]] \arrow[r, hook] \arrow[d, two heads] & Z_\alpha \arrow[d, two heads] \\
\vert\Delta_1^\circ[a,b]\vert \arrow[r, hook] & Z_1.
\end{tikzcd}
\]
We also define the variety $U_\alpha=Z_{\alpha}-\Delta_\alpha^\circ[[a,b]]$.

As before, let $\Lambda=\cl{O}[\![1+p\bb{Z}_p]\!]$ and $\Lambda^F=\cl{O}[\![1+p\cl{O}_{F,p}]\!]$. We define $\Lambda\otimes\Lambda^F$-module structure on $H^\ast_{\et}(Z_{\alpha,\overline{\bb{Q}}},\cl{O})$, $H^\ast_{\et}(\Delta_\alpha^\circ[[a,b]]_{\overline{\bb{Q}}},\cl{O})$ and $H^\ast_{\et}(U_{\alpha,\overline{\bb{Q}}},\cl{O})$ by letting a group-like element $[u_1]\otimes [u_2]$ act by $\langle u_1,u_2\rangle'=\langle u_1^{-1},u_2^{-1}\rangle$.

Let $[t_0,t_1]$ be an element in $I_\alpha$ mapping to $[a,b]$ in $I_1$. 
Then, we have the following commutative diagram of $\cl{O}[G_{\bb{Q}(\zeta_1)}]$-modules
\[
\begin{tikzcd}
H^3_{\et}(Z_{\alpha,\overline{\bb{Q}}},\cl{O})(2)(-\bm{\kappa}^{1/2}) \arrow[r, hook, dashed] \arrow[d, equal]  & \cl{E} \arrow[r, two heads, dashed] \arrow[d, hook, dashed] & \cl{O} \arrow[d, hook, "j"] \\
H^3_{\et}(Z_{\alpha,\overline{\bb{Q}}},\cl{O})(2)(-\bm{\kappa}^{1/2}) \arrow[r, hook] & H^3_{\et}(U_{\alpha,\overline{\bb{Q}}},\cl{O})(2)(-\bm{\kappa}^{1/2}) \arrow[r, two heads] & H^0_{\et}(\Delta_\alpha^\circ[[a,b]]_{\overline{\bb{Q}}},\cl{O})(-\bm{\kappa}^{1/2}),
\end{tikzcd}
\]
where
\begin{itemize}
    \item $\bm{\kappa}^{1/2}:G_{\bb{Q}(\zeta_1)}\rightarrow (\Lambda\otimes\Lambda^F)^\times$ is the character defined by $\sigma\mapsto [\epsilon_{\cyc}(\sigma)^{1/2}]\otimes[\epsilon_{\cyc}(\sigma)^{1/2}]$, with the square root taken in $1+p\bb{Z}_p$;
    \item $j:\cl{O}\lhook\joinrel\rightarrow H^0_{\et}(\Delta^\circ_\alpha[[a,b]]_{\overline{\bb{Q}}},\cl{O})(-\bm{\kappa}^{1/2})$ is the map sending $1\in\cl{O}$ to the cycle $\Delta_\alpha^\circ[t_0,t_1]$;
    \item the upper row is the pullback of the lower row via the map $j$.
\end{itemize}
To simplify notation, let $\lambda_\alpha$ denote the Atkin--Lehner map $\lambda_{Np^\alpha}$ introduced in Section~\ref{subsec:AL}, using the same notation both for $Y_1(p^\alpha)$ and for $S_1(p^\alpha)$. Then, we define the class
\[
\tilde{\kappa}_\alpha[t_0,t_1]\in H^1(\bb{Q}(\zeta_1),H^3_{\et}(Z_{\alpha,\overline{\bb{Q}}},\cl{O})(2)(-\bm{\kappa}^{1/2}))
\]
as the image of $1\in\cl{O}$ via the connecting map
\[
\cl{O}\longrightarrow H^1(\bb{Q}(\zeta_1),H^3_{\et}(Z_{\alpha,\overline{\bb{Q}}},\cl{O})(2)(-\bm{\kappa}^{1/2}))
\]
arising from the upper short exact sequence in the diagram, and we define the class
\[
\kappa_\alpha[t_0,t_1]\in H^1(\bb{Q}(\zeta_1),H^3_{\et}(Z_{\alpha,\overline{\bb{Q}}},\cl{O})(2)(-\bm{\kappa}^{1/2})(\langle\omega_N,\omega_N\rangle))
\]
as the image of $\tilde{\kappa}_\alpha[t_0,t_1]$ via the Atkin--Lehner map
\[
(\lambda_\alpha,\lambda_\alpha)_\ast : H^1(\bb{Q}(\zeta_1),H^3_{\et}(Z_{\alpha,\overline{\bb{Q}}},\cl{O})(2)(-\bm{\kappa}^{1/2}))\rightarrow H^1(\bb{Q}(\zeta_1),H^3_{\et}(Z_{\alpha,\overline{\bb{Q}}},\cl{O})(2)(-\bm{\kappa}^{1/2})(\langle\omega_N,\omega_N\rangle)),
\]
where $\omega_N:G_{\bb{Q}}\rightarrow \prod_{q\mid N} \bb{Z}_q^\times$ is defined by $\omega_N(\sigma)=\prod_{q\mid N} \epsilon_{q,\cyc}(\sigma)$.

\begin{lemma}
Let $\alpha\geq 1$ and let $[t_0',t_1']\in I_{\alpha+1}$ map to $[t_0,t_1]\in I_\alpha$. Then
\begin{align*}
(\varpi_2,\varpi_2)_{\ast}\kappa_{\alpha+1}[t_0',t_1'] &=p^3\kappa_{\alpha}[t_0,t_1], \\
(\varpi_1,\varpi_1)_\ast \kappa_{\alpha+1}[t_0',t_1'] &=(U_p',U_p')\kappa_{\alpha+1}[t_0,t_1].
\end{align*}
Moreover, the classes $\kappa_\alpha[t_0,t_1]$ satisfy the distribution relation
\[
\sum_{[t_0',t_1']}\kappa_{\alpha+1}[t_0',t_1']=(\varpi_2,\varpi_2)^\ast\kappa_\alpha[t_0,t_1],
\]
where the sum is taken over pairs $[t_0',t_1']\in I_{\alpha+1}$ which map to $[t_0,t_1]\in I_\alpha$.
\label{lemma:classes-relations}
\end{lemma}
\begin{proof}
This follows from Lemma~\ref{lemma:cycles-relations} taking into account the relations
\[
(\varpi_1,\varpi_1)_\ast\circ(\lambda_{\alpha+1},\lambda_{\alpha+1})=(\lambda_\alpha,\lambda_\alpha)_\ast\circ (\varpi_2,\varpi_2)_\ast,
\]
\[
(U_p',U_p')\circ(\lambda_\alpha,\lambda_\alpha)_\ast=(\lambda_\alpha,\lambda_\alpha)_\ast\circ (U_p,U_p).
\]
\end{proof}

Define
\[
H^3_{\et}(Z_{\infty,\overline{\bb{Q}}},\cl{O})=\varprojlim_{\alpha\geq 1} H^3_{\et}(Z_{\alpha,\overline{\bb{Q}}},\cl{O}),
\]
where the projective limit is taken with respect to the maps $(\varpi_1,\varpi_1)_\ast$, and let
\[
\bb{V}=e'H^3_{\et}(Z_{\infty,\overline{\bb{Q}}},\cl{O})(2)(-\bm{\kappa}^{1/2})(\langle\omega_N,\omega_N\rangle),
\]
where
\[
e'=\lim_{n\to\infty}(U_p',U_p')^{n!}
\]
is Hida's anti-ordinary projector. For convenience we also introduce the notation
\[
V_\alpha=H^3_{\et}(Z_{\alpha,\overline{\bb{Q}}},\cl{O})(2)(-\bm{\kappa}^{1/2}).
\]

Now, given an element $[t_0,t_1]\in I_\infty:=\varprojlim_{\alpha\geq 1}I_\alpha$, in view of Lemma~\ref{lemma:classes-relations} we define a class
\[
\kappa_\infty[t_0,t_1]=\varprojlim_{\alpha\geq 1} (U_p',U_p')^{-\alpha}e'\kappa_{\alpha}[t_0,t_1]\in H^1(\bb{Q}(\zeta_1),\bb{V}).
\]

Observe that we can canonically identify $I_1$ with the set of torsion elements in $\cl{O}_F^\times\times\cl{O}_F^{\times,\mathrm{tr}=0}$ and therefore we can define a canonical lift of $[a,b]$ to $I_\infty$, which we will denote in the same way. Let $\omega_0$ and $\omega_1$ be characters of $(\bb{Z}/p\bb{Z})^\times$ and $(\cl{O}_F/p\cl{O}_F)^\times$, respectively, taking values in $E^\times$ for some finite extension $E$ of $\bb{Q}_p$. Assume that there exists a character $\eta$ of $(\bb{Z}/p\bb{Z})^\times$ taking values in $E^\times$ (in fact in $\bb{Z}_p^\times$) such that $\omega_0\omega_1=\eta^2$ as characters of $(\bb{Z}/p\bb{Z})^\times$. The character $\eta$ is then determined by $\omega_0$ and $\omega_1$ up to the unique quadratic character of $(\bb{Z}/p\bb{Z})^\times$. Choose a trace-zero element $\delta\in (\cl{O}_F/p\cl{O}_F)^\times$. Define
\[
\kappa_{\infty}(\omega_0,\omega_1;\eta)=\frac{1}{\#I_1}\sum_{[a,b]\in I_1}\eta^{-1}(N_{F/\bb{Q}}(a)\delta b)\omega_0(\delta b)\omega_1(a)\kappa_{\infty}[a,b].
\]

We decompose the cyclotomic character $\epsilon_{\cyc}:G_\bb{Q}\rightarrow \bb{Z}_p^\times=\mu_{p-1}\times (1+p\bb{Z}_p)$ as $\epsilon_{\cyc}=\omega_{\cyc}\theta_{\cyc}$, with $\omega_{\cyc}$ and $\theta_{\cyc}$ taking values in $\mu_{p-1}$ and $1+p\bb{Z}_p$, respectively. Note that we can extend the character $\bm{\kappa}^{1/2}:G_{\bb{Q}(\zeta_1)}\rightarrow (\Lambda\otimes\Lambda^F)^\times$ to a character of $G_\bb{Q}$ by $\tau\mapsto [\theta_{\cyc}(\tau)^{1/2}]\otimes[\theta_{\cyc}(\tau)^{1/2}]$. Then, as a map of $\cl{O}[G_\bb{Q}]$-modules, the Atkin--Lehner correspondence $(\lambda_\alpha,\lambda_\alpha)$ actually yields a map
\[
(\lambda_\alpha,\lambda_\alpha)_\ast: H^1(\bb{Q}(\zeta_1), V_\alpha)\longrightarrow H^1(\bb{Q}(\zeta_1), V_\alpha)(\langle \omega_N\omega_{\cyc},\omega_N\omega_{\cyc}\rangle).
\]
In the following lemma we consider the class $\kappa_\infty(\omega_0,\omega_1;\eta)$ as a class in
\[
H^1(\bb{Q}(\zeta_1), \bb{V})(\langle \omega_{\cyc},\omega_{\cyc}\rangle).
\]

\begin{lemma}
For all $\sigma_m\in\Gal(\bb{Q}(\zeta_\infty)/\bb{Q})$,
\[
\sigma_m\kappa_\infty(\omega_0,\omega_1;\eta)=\eta(m)\kappa_\infty(\omega_0,\omega_1;\eta).
\]
For all $z_0\in\mu(\bb{Z}_p^\times)$, $z_1\in\mu(\cl{O}_{F,p}^\times)$,
\[
\langle z_0,z_1\rangle \kappa_\infty(\omega_0,\omega_1;\eta)=\omega_0(z_0)\omega_1(z_1^\sigma)\kappa_\infty(\omega_0,\omega_1;\eta).
\]
\end{lemma}
\begin{proof}
It follows easily from the definition of the class $\kappa_\infty(\omega_0,\omega_1;\eta)$ using Lemma~\ref{lemma:cycles-galoisaction}.
\end{proof}

Consider the idempotent
\[
e_{\omega_0,\omega_1}=\frac{1}{\# I_1}\sum_{(z_0,z_1)\in  \mu(\bb{Z}_p^\times)\times\mu(\cl{O}_{F,p}^\times)} \omega_0^{-1}(z_0)\omega_1^{-1}(z_1^\sigma)\langle z_0, z_1\rangle
\]
and let $\bb{V}_{\omega_0,\omega_1}=e_{\omega_0,\omega_1}\bb{V}$. As a consequence of the previous lemma, we have that
\[
\kappa_\infty(\omega_0,\omega_1;\eta)\in H^1(\bb{Q}(\zeta_1), \bb{V}_{\omega_0,\omega_1})(\bm{\eta}^2),
\]
where $\bm{\eta}=\eta\circ\epsilon_{\cyc}$,
and, after twisting, we can regard $\kappa_\infty(\omega_0,\omega_1;\eta)$ as the restriction of a class in
$H^1(\bb{Q}, \bb{V}_{\omega_0,\omega_1}(\bm{\eta}))$.
We use the same notation $\kappa_\infty(\omega_0,\omega_1;\eta)$ to denote this class.

\section{Euler system norm relations}\label{sec:norm-relations}

In this section, we construct a split anticyclotomic Euler system for the Asai  representation attached to ($p$-ordinary) Hilbert modular forms over real quadratic fields.

\subsection{CM Hida famlies}\label{subsec:CM-hida}

Let $\cl{K}$ be an imaginary quadratic field of discriminant $-D<0$ and let $\varepsilon_{\cl{K}}$ be the corresponding quadratic character. Assume that $p$ splits in $\cl{K}$ as 
\[
(p)=\frk{p}\overline{\frk{p}},
\]
with $\frk{p}$ the prime of $\cal{K}$ above $p$ induced by our fixed embedding $\iota_p:\overline{\bb{Q}}\hookrightarrow\overline{\bb{Q}}_p$. Assume also that $p$ does not divide the class number $h_{\cl{K}}$. Let $\psi$ be a Hecke character of $\cl{K}$ of conductor $\frk{c}$ coprime to $p$ and infinity type $(1-k,0)$ for some even integer $k\geq 2$, taking values in a finite extension $L/\cl{K}$. Let $\chi_\psi$ be the unique Dirichlet character modulo $N_{\cl{K}/\bb{Q}}(\frk{c})$ such that 
\[
\psi((n))=n^{k-1} \chi_\psi(n)
\]
for integers $n$ coprime to $N_{\cl{K}/\bb{Q}}(\frk{c})$. Put $N_\psi=N_{\cl{K}/\bb{Q}}(\frk{c})D$, and let $\theta_{\psi} \in S_k(N_{\psi}, \chi_\psi\varepsilon_{\cl{K}})$ be the theta series attached to $\psi$, i.e.,
\[
\theta_{\psi} = \sum_{(\frk{a}, \frk{c})=1} \psi(\frk{a}) q^{N_{\cl{K}/\bb{Q}}(\frk{a})}.
\]

Let $\frk{P}$ denote the prime of $L$ above $\frk{p}$ induced by $\iota_p$, put $E=L_{\frk{P}}$ and let $\cl{O}\subset E$ be the ring of integers. Let $\lambda$ denote the unique Hecke character of infinity type $(-1,0)$ and conductor $\frk{p}$ whose $p$-adic avatar $\lambda_\mathfrak{P}:\cal{K}^\times\backslash \bb{A}_{\cl{K},f}^{\times}\rightarrow E^\times$, defined by
\[
\lambda_{\frk{P}}(x)=x_{\frk{p}}^{-1} \lambda(x)
\]
for $x_{\frk{p}}$ the $\frk{p}$-component of $x\in\bb{A}_{\cl{K},f}$, factors through $\Gamma_\frk{p}$, the Galois group of the unique $\bb{Z}_p$-extension of $\cl{K}$ unramified outside $\frk{p}$. Then we can uniquely write $\psi=\psi_0\lambda^{k-1}$, with $\psi_0$ a ray class character of conductor dividing $\frk{c}\frk{p}$. Since $(\frk{c},p)=1$ and $k$ is even, it easily follows that $\psi$ is \emph{non-Eisenstein} and \emph{$p$-distinguished}, meaning that
\begin{equation}\label{eq:513}
\psi\vert_{\cl{O}_{\cl{K},\frk{p}}^\times}\not\equiv\omega\;({\mathrm{mod}}\;{\frk{P}}),
\end{equation}
where $\omega$ is the Teichm\"uller character (see \cite[Rem.~5.1.3]{LLZ0}. Letting $\psi_{\frk{P}}$ be the continuous $E$-valued character of $\cl{K}^{\times} \backslash \bb{A}_{\cl{K},\mathrm{f}}^{\times}$ defined by
\[
\psi_{\frk{P}}(x)=x_{\frk{p}}^{1-k} \psi(x),
\]
and viewing it as a character of $G_{\cl{K}}$ via the geometrically normalized Artin map, the $p$-adic representation (dual to Deligne's) attached to the eigenform $\theta_\psi$ is given by $\Ind_{\cl{K}}^\bb{Q} E(\psi_\frk{P}^{-1})$. 
Note that by (\ref{eq:513}), the associated residual representation is absolutely irreducible and $p$-distinguished.

Consider the $q$-expansion
\[
\Theta=\sum_{(\frk{a},\frk{cp})=1}[\frk{a}]q^{N_{\cl{K}/\bb{Q}}(\frk{a})}\in \cl{O}[\![H_{\frk{cp}^\infty}]\!][\![q]\!],
\]
where $H_{\frk{cp}^\infty}$ denotes the maximal pro-$p$ quotient of the ray class group of $\cl{K}$ of conductor $\frk{cp}^\infty$, and $[\frk{a}]$ is the image of $\frk{a}$ in $H_{\frk{cp}^\infty}$ under the Artin map. Since we assume that $p\nmid h_{\cl{K}}$, we can factor $H_{\frk{cp}^\infty}\cong H_\frk{c}\times \Gamma_\frk{p}$. Hence, we have $\Theta\in \cl{O}[H_\frk{c}]\otimes_{\cl{O}}\cl{O}[\![\Gamma_\frk{p}]\!][\![q]\!]$, and putting $\bar{\psi}_0:=\psi_0\vert_{H_\frk{c}}$ we can specialize this to
\begin{equation}\label{eq:CM-F}
\hf = \sum_{(\frk{a},\frk{cp})=1}\bar{\psi}_0([\frk{a}])[\frk{a}]_{\frk{p}}q^{N_{\cl{K}/\bb{Q}}(\frk{a})}\in \Lambda_{\hf}[\![q]\!],
\end{equation}
where $\Lambda_{\hf}=\cl{O}[\![\Gamma_{\frk{p}}]\!]$ and $[\frk{a}]_{\frk{p}}$ denotes the image of $\frk{a}$ in $\Gamma_\frk{p}$. We identify $\Gamma_{\frk{p}}$ with $\Gamma=1+p\bb{Z}_p$ via the isomorphism $\Gamma\cong\cl{O}_{\cl{K},\frk{p}}^{(1)}\rightarrow \Gamma_{\frk{p}}$ defined by $u\mapsto \mathrm{art}_{\frk{p}}(u)^{-1}$, where $\mathrm{art}_{\frk{p}}$ stands for the geometric local Artin map, and in this way we identify $\Lambda_{\hf}$ with $\Lambda$. We can therefore regard $\hf$ as a primitive Hida family of tame level $N_\psi$ and character $\chi_\psi\varepsilon_{\cl{K}}\omega^{k-2}$ passing through the ordinary $p$-stabilization of $\theta_\psi$.

Let $\Gamma_{\ac}$ be the Galois group of the anticyclotomic $\bb{Z}_p$-extension of $\cl{K}$; this can be identified with the anti-diagonal in $(1+p\bb{Z}_p)\times (1+p\bb{Z}_p)\cong \cl{O}_{\cl{K},\frk{p}}^{(1)}\times \cl{O}_{\cl{K},\overline{\frk{p}}}^{(1)}$ via the Artin map. Let 
\[
\kappa_{\ac}:\Gamma_{\ac}\longrightarrow \bb{Z}_p^\times
\]
be the character defined by mapping the element $((1+p)^{-1},(1+p))$ to $1+p$ and let $\bm{\kappa}_{\ac}:\Gamma_{\ac}\rightarrow \Lambda^\times$ be the character defined by mapping $((1+p)^{-1}, (1+p))$ to the group-like element $[1+p]$. We use the same notation for the corresponding characters of $G_\bb{Q}$.

For each positive integer $n$, let $\cl{K}[n]$ denote the maximal $p$-subextension of the ring class field of $\cl{K}$ of conductor $n$ and let $R_n=\Gal(\cl{K}[n]/\cl{K})$. Also, for $n$ coprime to $pN_\psi$ and $\alpha>0$, let $Y_{\alpha,n}$ denote the affine modular curve of level $V^1(N_\psi p^\alpha)\cap V_0(n^2)$, and put
\[
H^1_{\et}(Y_{\infty,n,\overline{\bb{Q}}},\cl{O})=\varprojlim_\alpha H^1_{\et}(Y_{\alpha,n,\overline{\bb{Q}}},\cl{O}),
\]
with the projective limit defined with respect to the maps $\varpi_{1\ast}$. We also let
\[
\bb{T}'(N_\psi(n^2))=e'\frk{h}_{\bb{Q}}(V^1(N_\psi)\cap V_0(n^2);\cl{O})
\]
denote the big anti-ordinary Hecke algebra acting on $e'H^1_{\et}(Y_{\infty,n,\overline{\bb{Q}}},\cl{O})$, and we regard $\bb{T}'(N_\psi(n^2))$ as a $\Lambda$-module by sending the group-like element $[z]\in \Lambda$ to $\langle z\rangle'\in \bb{T}'(N_\psi(n^2))$.

Let $\Gamma_{n,\frk{p}}$ denote the Galois group over $\cl{K}$ of the compositum of $\cl{K}[n]$ and the unique $\bb{Z}_p$-extension of $\cl{K}$ unramified outside $\frk{p}$, and note that $\Gamma_{n,\frk{p}}\cong R_n\times\Gamma_\frk{p}$. 

\begin{prop}
There exists a homomorphism $\phi_n:\bb{T}'(N_\psi(n^2))\rightarrow \Lambda[R_n]\cong \cl{O}[\![\Gamma_{n,\frk{p}}]\!]$ defined on generators by
\[
\phi_n(T_q')=\omega(q)^{k/2-1}q^{1-k/2}\sum_{\frk{q}}(\kappa_{\ac}^{2-k}\psi)(\frk{q})[\frk{q}],
\]
for every rational prime $q$, where the sums runs over ideals of $\cl{O}_K$ coprime to $n\frk{c}\frk{p}$ of norm $q$; and
\[
\phi_n(\langle \varpi_q^{-1}\rangle')=(\omega^{k-2}\chi_\psi\varepsilon_\cl{K})(q)[(q)]
\]
for all rational prime coprime to $pN_\psi n$.
\end{prop}
\begin{proof}
    This follows from \cite[Prop.~3.2.1]{LLZ} as in \cite[Lemma~3.5]{ACR}.
\end{proof}

Let $n$ be a positive integer coprime to $pN_\psi$ and let $q$ be a rational prime coprime to $pN_\psi$. Let $\pi_{11,\ast}$, $\pi_{12,\ast}$ and $\pi_{22,\ast}$ denote the degeneracy maps
\[
\pi_{ij,\ast}=\pi_{i,q,\ast}\circ \pi_{j,q,\ast}: H^1_{\et}(Y_{\infty,nq,\overline{\bb{Q}}},\cl{O}) \longrightarrow H^1_{\et}(Y_{\infty,n,\overline{\bb{Q}}},\cl{O}).
\]
Following \cite[\S~3.3]{LLZ}, we define norm maps
\[
\cl{N}_{n}^{nq}:\Lambda[R_{nq}]\otimes_{\phi_{nq}} H^1_{\et}(Y_{\infty,nq,\overline{\bb{Q}}},\cl{O})\longrightarrow \Lambda[R_{n}]\otimes_{\phi_n} H^1_{\et}(Y_{\infty,n,\overline{\bb{Q}}},\cl{O})
\]
by the formulae:
\begin{itemize}
    \item if $q\mid n$,
    \[
    \cl{N}_{n}^{nq}=1\otimes \pi_{11,\ast};
    \]
    \item if $q\nmid n$ and $q$ splits in $\cl{K}$ as $(q)=\frk{q}\overline{\frk{q}}$,
    \begin{align*}
    \cl{N}_{n}^{nq}&=1\otimes \pi_{11,\ast}-\omega^{(k-2)/2}(q)\left(\frac{\kappa_{\ac}^{2-k}\psi_{\frk{P}}(\Fr_\frk{q}^{-1})[\frk{q}]}{q^{k/2}}+\frac{\kappa_{\ac}^{2-k}\psi_{\frk{P}}(\Fr_{\overline{\frk{q}}}^{-1})[\overline{\frk{q}}]}{q^{k/2}}\right)\otimes \pi_{12,\ast}\\
    &+\frac{\chi_\psi(q)\omega^{k-2}(q)}{q}[(q)]\otimes \pi_{22,\ast};
    \end{align*}
    \item if $q\nmid n$ and $q$ is inert in $\cl{K}$,
    \[
    \cl{N}_{n}^{nq}=1\otimes \pi_{11,\ast}-\frac{\chi_\psi(q)\omega^{k-2}(q)}{q}[(q)]\otimes \pi_{22,\ast}.
    \]
\end{itemize}

More generally, if $n'=nr$, where $r$ is coprime to $pN_\psi$, we define $\cl{N}_{n}^{n'}$ by composing the previously defined norm maps in the natural way.

Let $\bm{\kappa}_0^{1/2}:G_\bb{Q}\rightarrow \Lambda^\times$ be the character defined by $\tau\mapsto[\theta_{\cyc}(\tau)^{1/2}]$ and let $\bm{\eta}_0:G_\bb{Q}\rightarrow E^\times$ be the character defined by $\tau\mapsto \omega_{\cyc}(\tau)^{1-k/2}$.

\begin{lemma}\label{lemma:LLZ}
Let $B$ be the set of positive integers coprime to $pN_\psi$. Then there is a family of $\Lambda[G_\bb{Q}]$-equivariant isomorphisms
\begin{equation*}
\nu_{n}:\Lambda[R_{n}]\otimes_{\phi_n} H^1_{\et}(Y_{\infty,n,\overline{\bb{Q}}},\cl{O})(\bm{\eta}_0\bm{\kappa}_0^{-1/2}) \longrightarrow  \Ind_{\cl{K}[n]}^\bb{Q}\Lambda(\psi_\frk{P}^{-1}\kappa_{\ac}^{k-2}\bm{\kappa}_{\ac}^{-1})(-k/2)
\end{equation*}
for all $n\in B$, such that for $n\mid n'$ the diagram
\[
\begin{tikzcd}
\Lambda[R_{n'}]\otimes_{\phi_{n'}} H^1_{\et}(Y_{\infty,n',\overline{\bb{Q}}},\cl{O})(\bm{\eta}_0\bm{\kappa}_0^{-1/2}) \arrow[r, "\nu_{n'}"] \arrow[d, "\cl{N}_n^{n'}"] & \Ind_{\cl{K}[n']}^\bb{Q}\Lambda(\psi_\frk{P}^{-1}\kappa_{\ac}^{k-2}\bm{\kappa}_{\ac}^{-1})(-k/2) \arrow[d, "\mathrm{Norm}"] \\
\Lambda[R_{n}]\otimes_{\phi_n} H^1_{\et}(Y_{\infty,n,\overline{\bb{Q}}},\cl{O})(\bm{\eta}_0\bm{\kappa}_0^{-1/2}) \arrow[r, "\nu_n"] & \Ind_{\cl{K}[n]}^\bb{Q}\Lambda(\psi_\frk{P}^{-1}\kappa_{\ac}^{k-2}\bm{\kappa}_{\ac}^{-1})(-k/2)
\end{tikzcd}
\]
commutes.
\end{lemma}
\begin{proof}
    This follows from \cite[Cor.~5.2.6]{LLZ} as in \cite[Cor.~3.6]{ACR}.
\end{proof}

\subsection{Norm relations}

Let $\psi$ be a Hecke character of $\cal{K}$ as above, and let $\hg\in \bar{\mathbb S}_F^{\ord}(\frk{N}_g,\chi_{\hg};\bb{I}_{\hg})$ be a Hida family passing through the ordinary $p$-stabilization of an ordinary Hilbert eigenform $g\in S_l(\frk{N}_g,\chi_g;\bb{C})$. Assume that 
\[
\chi_\psi\varepsilon_\cl{K}\chi_g\vert_{\bb{A}_{\bb{Q},f}^\times}=1.
\]
Let $N=N_\psi N_{F/\bb{Q}}(\frk{N}_g)$, and put 
\[
K=V^1(N)\leq \mathrm{GL}_2(\mathbb{A}_{F,f}),\quad U=K\cap \mathrm{GL}_2(\mathbb{A}_{\mathbb{Q},f}).
\]
For each positive integer $\alpha$, let $S_{\alpha}$ be the Hilbert modular surface of level $K^1(p^\alpha)$ and let $Y_\alpha$ be the modular curve of level $U^1(p^\alpha)$. In Section~\ref{sec:cycles}, we constructed codimension-2 cycles
\[
\Delta_\alpha[t_0,t_1]\in\ch^2(Y_\alpha\times S_\alpha)(\bb{Q}(\zeta_\alpha)).
\]

For any positive integer $m$ coprime to $pN$, let $Y_{10}(p^\alpha,m)$ denote the modular curve of level $U^1(p^\alpha)\cap U_0(m)$, let $S_{10}(p^\alpha,m)$ denote the Hilbert modular surface of level $K^1(p^\alpha)\cap K_0(m)$ and let $\tilde{\Delta}_{\alpha,m}[t_0,t_1]\in \CH(Y_{10}(p^\alpha,m)\times S_{10}(p^\alpha,m))$ be the cycle constructed in Section~\ref{sec:cycles} taking as base level $K_0(m)=K\cap V_0(m)$ instead of $K$. Recall that the construction requires the choice of a connected component $\bb{Y}_m(p^\alpha)$ of the modular curve $Y_m(p^\alpha)$ of level $U(p^\alpha)\cap U_0(m)$. The complex points of $Y_m(p^\alpha)$ are given by
\[
Y_m(p^\alpha)(\bb{C})= \mathrm{GL}_2(\bb{Q})^+\backslash \cl{H}\times \mathrm{GL}_2(\bb{A}_{\bb{Q},f})/U(p^\alpha)\cap U_0(m),
\]
and we choose $\bb{Y}_m(p^\alpha)$ to be the image of the embedding $\Gamma_{\alpha,m}\backslash\cl{H}\hookrightarrow Y_m(p^\alpha)(\bb{C})$ taking $z$ to $[z,1]$, where $\Gamma_{\alpha,m}=\mathrm{GL}_2(\bb{Q})^+\cap U(p^\alpha)\cap U_0(m)$.

\begin{lemma}\label{lem:deg-maps}
    Let $m$ be a positive integer coprime to $pN$ and let $q$ be a rational prime coprime to $pN$ which splits in $F$ as $(q)=\frk{q}_1\frk{q}_2$. Then
    \begin{align*}
    &(\pi_{1,q},\pi_{1,\frk{q}_1}\pi_{2,\frk{q}_2})_\ast \tilde{\Delta}_{\alpha,mq}[t_0,t_1]=(1,T_{\frk{q}_2}) \tilde{\Delta}_{\alpha,m}[t_0,t_1];\\
    & (\pi_{1,q},\pi_{2,\frk{q}_1}\pi_{2,\frk{q}_2})_\ast \tilde{\Delta}_{\alpha,mq}[t_0,t_1]=(T_q',1) \tilde{\Delta}_{\alpha,m}[q^{-1}t_0,q^{-1}t_1]; \\
    & (\pi_{2,q},\pi_{1,\frk{q}_1}\pi_{1,\frk{q}_2})_\ast \tilde{\Delta}_{\alpha,mq}[t_0,t_1]=(T_q,1) \tilde{\Delta}_{\alpha,m}[t_0,t_1]; \\ 
    &(\pi_{2,q},\pi_{1,\frk{q}_1}\pi_{2,\frk{q}_2})_\ast \tilde{\Delta}_{\alpha,mq}[t_0,t_1]=(1, T_{\frk{q}_1}') \tilde{\Delta}_{\alpha,m}[q^{-1}t_0,q^{-1}t_1].
    \end{align*}
    If $q$ is coprime to $m$ we have
    \begin{align*}
    & (\pi_{1,q},\pi_{1,\frk{q}_1}\pi_{1,\frk{q}_2})_\ast \tilde{\Delta}_{\alpha,mq}[t_0,t_1]=(q+1) \tilde{\Delta}_{\alpha,m}[t_0,t_1]; \\ 
    & (\pi_{2,q},\pi_{2,\frk{q}_1}\pi_{2,\frk{q}_2})_\ast \tilde{\Delta}_{\alpha,mq}[t_0,t_1]=(q+1)\tilde{\Delta}_{\alpha,m}[q^{-1}t_0,q^{-1}t_1];
    \end{align*}
    otherwise, if $q\mid m$, we have
    \begin{align*}
    & (\pi_{1,q},\pi_{1,\frk{q}_1}\pi_{1,\frk{q}_2})_\ast \tilde{\Delta}_{\alpha,mq}[t_0,t_1]=q \tilde{\Delta}_{\alpha,m}[t_0,t_1]; \\ 
    & (\pi_{2,q},\pi_{2,\frk{q}_1}\pi_{2,\frk{q}_2})_\ast \tilde{\Delta}_{\alpha,mq}[t_0,t_1]=q \tilde{\Delta}_{\alpha,m}[q^{-1}t_0,q^{-1}t_1].
    \end{align*}
\end{lemma}
\begin{proof}
    Let $(v_0,v_1)\in \Sigma_\alpha[t_0,t_1]$. Choose matrices $\delta_0\in\SL_2(\bb{Z}_p)$ and $\delta_1\in\SL_2(\cl{O}_{F,p})$ such that the reduction modulo $p^\alpha$ of their bottom row is given by $v_0$ and $v_1$, respectively. Let $a\in \mathrm{GL}_2(\widehat{\bb{Z}})$ be the element with local components $a_v=\mathbbm{1}_2$ for $v\neq p$ and $a_p=\delta_0^{-1}$. Let $b\in \mathrm{GL}_2(\widehat{\cl{O}}_F)$ be the element with local components $a_v=\mathbbm{1}_v$ for $v\nmid p$ and $b_p=\delta_1^{-1}$. Then $\tilde{\Delta}_{\alpha,m}[t_0,t_1]$ is the scheme-theoretic image of $\bb{Y}_m(p^\alpha)$ by the morphism of varieties
    \begin{align*}
        \varphi_{(v_0,v_1)}: Y_m(p^\alpha) &\longrightarrow Y_{10}(p^\alpha,m)\times S_{10}(p^\alpha,m) \\
        [x,g] &\longmapsto ([x,ga],[x,gb]),
    \end{align*}
    and similarly for $\tilde{\Delta}_{\alpha,mq}$.

    Consider the following commutative diagram
    \[
    \begin{tikzcd}[column sep=large]
    \bb{Y}_m(p^{\alpha}) \arrow[r, hook, "\varphi_{(v_0,v_1)}"] \arrow[rd, hook, "\phi_{(v_0,v_1)}"'] & Y_{10}(p^{\alpha},mq)\times S_{10}(p^{\alpha},mq) \arrow[d, "{(1,\pi_{1,\frk{q}_1}\pi_{1,\frk{q}_2})}"] \\
    & Y_{10}(p^{\alpha},mq)\times S_{10}(p^{\alpha},m),
    \end{tikzcd}
    \]
    where $\phi_{(v_0,v_1)}$ is simply defined as the composition of the other two maps. Note that
    \[
    \phi_{(v_0,v_1)}(\bb{Y}_{mq}(p^\alpha))=(1,\pi_{1,\frk{q}_1}\pi_{1,\frk{q}_2})_\ast \varphi_{(v_0,v_1)}(\bb{Y}_m(p^\alpha))=(1,\pi_{1,\frk{q}_1}\pi_{1,\frk{q}_2})_\ast \tilde{\Delta}_{\alpha,mq}[t_0,t_1].
    \]
    We also have the commutative diagram
    \[
    \begin{tikzcd}[column sep=large]
    \bb{Y}_{mq}(p^{\alpha}) \arrow[r, hook, "\phi_{(v_0,v_1)}"] \arrow[d, "\pi_{1,q}"] & Y_{10}(p^{\alpha},mq)\times S_{10}(p^{\alpha},m) \arrow[d, "{(\pi_{1,q},1)}"] \\
    \bb{Y}_m(p^\alpha) \arrow[r, hook, "\varphi_{(v_0,v_1)}"] & Y_{10}(p^{\alpha},m)\times S_{10}(p^{\alpha},m),
    \end{tikzcd}
    \]
    where the vertical maps are both finite of degree $q$, if $q\mid m$, or $q+1$, if $q\nmid m$. It follows that
    \[
    (\pi_{1,q},1)^\ast \tilde{\Delta}_{\alpha,m}[t_0,t_1]=\phi_{(v_0,v_1)}(\bb{Y}_{mq}(p^\alpha))=(1,\pi_{1,\frk{q}_1}\pi_{1,\frk{q}_2})_\ast \tilde{\Delta}_{\alpha,mq}[t_0,t_1].
    \]
    Therefore
    \[
    (\pi_{2,q},\pi_{1,\frk{q}_1}\pi_{1,\frk{q}_2})_\ast \tilde{\Delta}_{\alpha,mq}[t_0,t_1] = (\pi_{2,q},1)_\ast (\pi_{1,q},1)^\ast \tilde{\Delta}_{\alpha,m}[t_0,t_1] = (T_q,1)\tilde{\Delta}_{\alpha,m}[t_0,t_1]
    \]
    and
    \[
    (\pi_{1,q},\pi_{1,\frk{q}_1}\pi_{1,\frk{q}_2})_\ast \tilde{\Delta}_{\alpha,mq}[t_0,t_1] = (\pi_{1,q},1)_\ast (\pi_{1,q},1)^\ast \tilde{\Delta}_{\alpha,m}[t_0,t_1] = \deg(\pi_{1,q})\tilde{\Delta}_{\alpha,m}[t_0,t_1].
    \]
    A similar argument shows that
    \[
    (\pi_{1,q},\pi_{1,\frk{q}_1}\pi_{2,\frk{q}_2})_\ast \tilde{\Delta}_{\alpha,mq}[t_0,t_1]=(1,T_{\frk{q}_2}) \tilde{\Delta}_{\alpha,m}[t_0,t_1].
    \]

    Now consider the commutative diagram
    \[
    \begin{tikzcd}[column sep=large]
    \bb{Y}_m(p^{\alpha}) \arrow[r, hook, "\varphi_{(v_0,v_1)}"] \arrow[rd, hook, "\psi_{(v_0,v_1)}"'] & Y_{10}(p^{\alpha},mq)\times S_{10}(p^{\alpha},mq) \arrow[d, "{(1,\pi_{2,\frk{q}_1}\pi_{2,\frk{q}_2})}"] \\
    & Y_{10}(p^{\alpha},mq)\times S_{10}(p^{\alpha},m),
    \end{tikzcd}
    \]
    where $\psi_{(v_0,v_1)}$ is simply defined as the composition of the other two maps. Note that
    \[
    \psi_{(v_0,v_1)}(\bb{Y}_{mq}(p^\alpha))=(1,\pi_{2,\frk{q}_1}\pi_{2,\frk{q}_2})_\ast \varphi_{(v_0,v_1)}(\bb{Y}_m(p^\alpha))=(1,\pi_{2,\frk{q}_1}\pi_{2,\frk{q}_2})_\ast \tilde{\Delta}_{\alpha,mq}[t_0,t_1].
    \]
    Let $\bb{Y}_m(p^\alpha)_{\eta_q}$ be the image of the embedding $\Gamma^{\eta_q}_{\alpha,m}\backslash\cl{H}\hookrightarrow Y_m(p^\alpha)(\bb{C})$ taking $z$ to $[z,\eta_q]$, where $\Gamma^{\eta_q}_{\alpha,m}=\mathrm{GL}_2(\bb{Q})^+\cap \eta_q(U(p^\alpha)\cap U_0(m))\eta_q^{-1}$. Then, we have the commutative diagram
    \[
    \begin{tikzcd}[column sep=large]
    \bb{Y}_{mq}(p^{\alpha}) \arrow[r, hook, "\psi_{(v_0,v_1)}"] \arrow[d, "\pi_{2,q}"] & Y_{10}(p^{\alpha},mq)\times S_{10}(p^{\alpha},m) \arrow[d, "{(\pi_{2,q},1)}"] \\
    \bb{Y}_m(p^\alpha)_{\eta_q} \arrow[r, hook, "\varphi_{(v_0,v_1)}"] & Y_{10}(p^{\alpha},m)\times S_{10}(p^{\alpha},m),
    \end{tikzcd}
    \]
    where the vertical maps are both finite of degree $q$, if $q\mid m$, or $q+1$, if $q\nmid m$. It is easy to see that
    \[
    \varphi_{(v_0,v_1)}(\bb{Y}_m(p^\alpha)_{\eta_q})=\tilde{\Delta}_{\alpha,m}[q^{-1}t_0,q^{-1}t_1].
    \]
    It follows that
    \[
    (\pi_{2,q},1)^\ast  \tilde{\Delta}_{\alpha,m}[q^{-1}t_0,q^{-1}t_1]=\psi_{(v_0,v_1)}(\bb{Y}_{mq}(p^\alpha))=(1,\pi_{2,\frk{q}_1}\pi_{2,\frk{q}_2})_\ast \tilde{\Delta}_{\alpha,mq}[t_0,t_1].
    \]
    Therefore
    \begin{align*}
    (\pi_{1,q},\pi_{2,\frk{q}_1}\pi_{2,\frk{q}_2})_\ast \tilde{\Delta}_{\alpha,mq}[t_0,t_1] &= (\pi_{1,q},1)_\ast (\pi_{2,q},1)^\ast \tilde{\Delta}_{\alpha,m}[q^{-1}t_0,q^{-1}t_1] \\ 
    &= (T_q',1) \tilde{\Delta}_{\alpha,m}[q^{-1}t_0,q^{-1}t_1]
    \end{align*}
    and
    \begin{align*}
    (\pi_{2,q},\pi_{2,\frk{q}_1}\pi_{2,\frk{q}_2})_\ast \tilde{\Delta}_{\alpha,mq}[t_0,t_1] &= (\pi_{2,q},1)_\ast (\pi_{2,q},1)^\ast \tilde{\Delta}_{\alpha,m}[q^{-1}t_0,q^{-1}t_1] \\ 
    &= \deg(\pi_{2,q})\tilde{\Delta}_{\alpha,m}[q^{-1}t_0,q^{-1}t_1].
    \end{align*}
    A similar argument shows that
    \[
    (\pi_{2,q},\pi_{1,\frk{q}_1}\pi_{2,\frk{q}_2})_\ast \tilde{\Delta}_{\alpha,mq}[t_0,t_1]=(1, T_{\frk{q}_1}') \tilde{\Delta}_{\alpha,m}[q^{-1}t_0,q^{-1}t_1].
    \]
\end{proof}

For each rational prime $q$ coprime to $pN$ which splits in $F$ we fix a factorization $(q)=\frk{q}_1\frk{q}_2$ in $F$. Then, if $n$ is a squarefree product of such primes, we put $\frk{n}_1=\prod_{q\mid n}\frk{q}_1$ and $\frk{n}_2=\prod_{q\mid n}\frk{q}_2$ and define
\[
\Delta_{\alpha,n}[t_0,t_1]=(1,\pi_{1,\frk{n}_1^2}\pi_{2,\frk{n}_2^2})_\ast\tilde{\Delta}_{\alpha,n^2}[t_0,t_1]\in\ch^2(Y_{\alpha,n}\times S_\alpha)(\bb{Q}(\zeta_\alpha)),
\]
where $Y_{\alpha,n}$ denotes the modular curve of level $V^1(Np^\alpha)\cap V_0(n^2)$.

\begin{lemma}\label{lemma:norm-relations}
Let $n$ be as above and let $q$ be a rational prime coprime to $pNn$ and which splits in $F$ as $(q)=\frk{q}_1\frk{q}_2$. Then, the following relations hold:
\begin{align*}
(\pi_{1,q^2},1)_\ast \Delta_{\alpha,nq}[t_0,t_1]&=\left\{(1,T_{\frk{q}_2}^2)-(q+1)(1,\langle \varpi_{\frk{q}_2}^{-1}\rangle)\right\}\Delta_{\alpha,n}[t_0,t_1],\\
(\pi_{1,q}\pi_{2,q},1)_\ast \Delta_{\alpha,nq}[t_0,t_1]&=(1,\langle \varpi_{\frk{q}_1}\rangle T_{\frk{q}_1}T_{\frk{q}_2})\Delta_{\alpha,n}[q^{-1}t_0,q^{-1}t_1]-(\langle\varpi_q\rangle T_q,\langle \varpi_{\frk{q}_1}\rangle)\Delta_{\alpha,n}[q^{-2}t_0,q^{-2}t_1],\\
(\pi_{2,q^2},1)_\ast \Delta_{\alpha,nq}[t_0,t_1]&=\left\{(1,\langle \varpi_{\frk{q}_1}^2\rangle T_{\frk{q}_1}^2)-(q+1)(\langle 1,\langle \varpi_{\frk{q}_1}\rangle)\right\}\Delta_{\alpha,n}[q^{-2}t_0,q^{-2}t_1].
\end{align*}
\end{lemma}
\begin{proof}
After Lemma~\ref{lem:deg-maps}, the proof is essentially the same as that of \cite[Lem.~4.4]{ACR}. We include the computation that yields the first relation in the statement for the convenience of the reader:
\begin{align*}
    (\pi_{1,q^2},1)_\ast &\Delta_{\alpha,nq}[t_0,t_1] = (1,\pi_{1,\frk{n}_1^2}\pi_{2,\frk{n}_2^2})_\ast(\pi_{1,q^2},\pi_{1,\frk{q}_1^2}\pi_{2,\frk{q}_2^2})_\ast \tilde{\Delta}_{\alpha,n^2q^2}[t_0,t_1] \\
    &= (1,\pi_{1,\frk{n}_1^2}\pi_{2,\frk{n}_2^2})_\ast(\pi_{1,q},\pi_{1,\frk{q}_1}\pi_{2,\frk{q}_2})_\ast(\pi_{1,q},\pi_{1,\frk{q}_1}\pi_{2,\frk{q}_2})_\ast\tilde{\Delta}_{\alpha,n^2q^2}[t_0,t_1] \\
    &= (1,\pi_{1,\frk{n}_1^2}\pi_{2,\frk{n}_2^2})_\ast(\pi_{1,q},\pi_{1,\frk{q}_1}\pi_{2,\frk{q}_2})_\ast (1, T_{\frk{q}_2})\tilde{\Delta}_{\alpha,n^2q}[t_0,t_1] \\
    &= (1,\pi_{1,\frk{n}_1^2}\pi_{2,\frk{n}_2^2})_\ast \left\{(1,T_{\frk{q}_2})(\pi_{1,q},\pi_{1,\frk{q}_1}\pi_{2,\frk{q}_2})_\ast-(1,\langle \varpi_{\frk{q}_2}\rangle^{-1})(\pi_{1,q},\pi_{1,\frk{q}_1}\pi_{1,\frk{q}_2})_\ast \right\}\tilde{\Delta}_{\alpha,n^2q}[t_0,t_1] \\
    &= (1,\pi_{1,\frk{n}_1^2}\pi_{2,\frk{n}_2^2})_\ast \left\{(1,T_{\frk{q}_2}^2)-(q+1)(1,\langle \varpi_{\frk{q}_2}\rangle^{-1})\right\}\tilde{\Delta}_{\alpha,n^2}[t_0,t_1] \\
    &= \left\{(1,T_{\frk{q}_2}^2)-(q+1)(1,\langle \varpi_{\frk{q}_2}\rangle^{-1})\right\}\Delta_{\alpha,n}[t_0,t_1],
\end{align*}
using Lemma~\ref{lem:deg-maps} for the third and fifth equalities.
\end{proof}


Let $\omega:(\bb{Z}/p\bb{Z})^\times\rightarrow E^\times$ denote the Teichm\"uller character. We define the characters
\begin{gather*}
  \omega_0=\omega^{2-k}:(\bb{Z}/p\bb{Z})^\times\longrightarrow E^\times, \\
  \omega_1=\omega^{2-l}\circ N_{F/\bb{Q}}:(\cl{O}_{F}/p\cl{O}_{F})^\times\longrightarrow E^\times
\end{gather*}
and the characters
\begin{gather*}
  \eta_0=\omega^{1-k/2}:(\bb{Z}/p\bb{Z})^\times\longrightarrow E^\times, \\
  \eta_1=\omega^{2-l}:(\bb{Z}/p\bb{Z})^\times\longrightarrow E^\times,
\end{gather*}
so that $\eta_0^2=\omega_0$ and $\eta_1^2=\omega_1$ as characters of $(\bb{Z}/p\bb{Z})^\times$. Let $\eta=\eta_0\eta_1$. Let $\theta:\bb{Z}_p^\times\rightarrow 1+p\bb{Z}_p$ denote the projection corresponding to the canonical decomposition $\bb{Z}_p^\times\cong \mu_{p-1}\times (1+p\bb{Z}_p)$. Let $\kappa_0^{1/2}:\bb{Z}_p^\times\rightarrow \Lambda^\times$ be the character defined by $u\mapsto [\theta(u)^{1/2}]$ and let $\kappa_1^{1/2}:\bb{Z}_p^\times\rightarrow (\Lambda^F)^\times$ be the character defined by $u\mapsto[\theta(u)^{1/2}]_F$. Let $\kappa^{1/2}:\bb{Z}_p^\times\rightarrow (\Lambda\otimes\Lambda^F)^\times$ be the character defined by $u\mapsto [\theta(u)^{1/2}]\otimes [\theta(u)^{1/2}]_F$. We also define the $G_{\bb{Q}}$-characters 
\[
\bm{\eta}_0=\eta_0\circ \epsilon_{\cyc},\quad \bm{\eta}_1=\eta_1\circ\epsilon_{\cyc},\quad \bm{\eta}=\eta\circ\epsilon_{\cyc},
\]
and likewise
\[
\bm{\kappa}_0^{1/2}=\kappa_0^{1/2}\circ \epsilon_{\cyc},\quad \bm{\kappa}_1^{1/2}=\kappa_1^{1/2}\circ\epsilon_{\cyc},\quad  \bm{\kappa}^{1/2}=\kappa^{1/2}\circ\epsilon_{\cyc}.
\]
%
%
Similarly to what we did in Section~\ref{sec:cycles}, for each positive integer $n$ we define
\[
\bb{V}_{\omega_0,\omega_1,n}=e_{\omega_0,\omega_1}e'\varprojlim_\alpha H^3_{\et}(Y_{\alpha,n,\overline{\bb{Q}}}\times S_{\alpha,\overline{\bb{Q}}},\cl{O})(2)(\bm{\kappa}^{-1/2})(\langle \omega_N,\omega_N\rangle),
\]
with projective limit taken with respect to the maps $(\varpi_1,\varpi_1)_\ast$. It follows from the K\"unneth decomposition that there exists a projection
\[
\bb{V}_{\omega_0,\omega_1,n}\longrightarrow \bb{V}_{\omega_0,n}\hat{\otimes}_{\cl{O}} \bb{V}_{\omega_1},
\]
where
\[
\bb{V}_{\omega_0,n}=e_{\omega_0}e'\varprojlim_\alpha H^1_{\et}(Y_{\alpha,n,\overline{\bb{Q}}},\cl{O}(1))(\bm{\kappa}_0^{-1/2})(\langle \omega_N\rangle),
\]
and
\[
\bb{V}_{\omega_1}=e_{\omega_1}e'\varprojlim_\alpha H^2_{\et}(S_{\alpha,\overline{\bb{Q}}},\cl{O}(1))(\bm{\kappa}_1^{-1/2})(\langle \omega_N\rangle).
\]
As in Section~\ref{sec:cycles}, we can use the cycles $\Delta_{\alpha,n}[t_0,t_1]$ to construct cohomology classes
\[
\kappa_{\infty,n}(\omega_0,\omega_1;\eta)\in H^1(\bb{Q},\bb{V}_{\omega_0,n}(\bm{\eta}_0)\hat{\otimes}_{\cl{O}} \bb{V}_{\omega_1}(\bm{\eta}_1)).
\]



\begin{prop}\label{prop:norm-relations1}
Let $n$ be as above and let $q$ be a rational prime coprime to $pNn$ and which splits in $F$ as $(q)=\frk{q}_1\frk{q}_2$. Then, the following relations hold:
\begin{align*}
(\pi_{2,q^2}, 1)_\ast \kappa_{\infty,nq}(\omega_0,\omega_1;\eta)&=\left\{(1,(T_{\frk{q}_2}')^2)-(q+1)(1,\langle \varpi_{\frk{q}_2}^{-1}\rangle')\right\}\kappa_{\infty,n}(\omega_0,\omega_1;\eta) \\
(\pi_{1,q}\pi_{2,q},1)_\ast \kappa_{\infty,nq}(\omega_0,\omega_1;\eta)&=\big\{\eta^{-1}(q)\kappa^{1/2}(q)(1,\langle \varpi_{\frk{q}_1}\rangle' T_{\frk{q}_1}'T_{\frk{q}_2}') \\
&\quad\quad\quad\quad -\omega_0^{-1}(q)\omega_1^{-1}(q)\kappa(q)(\langle \varpi_q\rangle' T_q',\langle \varpi_{\frk{q}_1}\rangle')\big\}\kappa_{\infty,n}(\omega_0,\omega_1;\eta) \\
(\pi_{1,q^2},1)_\ast \kappa_{\infty,nq}(\omega_0,\omega_1;\eta)&=\omega_0^{-1}(q)\omega_1^{-1}(q)\kappa(q)\left\{(1,\langle \varpi_{\frk{q}_1}^2\rangle' (T_{\frk{q}_1}')^2)-(q+1)(1,\langle \varpi_{\frk{q}_1}\rangle')\right\}\kappa_{\infty,n}(\omega_0,\omega_1;\eta)
\end{align*}
\end{prop}

\begin{proof}
This is a direct consequence of Lemma~\ref{lemma:norm-relations}, noting that the Atkin--Lehner map $(\lambda_\alpha,\lambda_\alpha)$ used in the construction of the classes $\kappa_{\infty,n}(\omega_0,\omega_1;\eta)$ interchanges the degeneracy maps $\pi_1$ and $\pi_2$.
\end{proof}

Let $\As(\bb{V}_\hg)^\dagger=\As(\bb{V}_\hg)(\bm{\eta}_1\bm{\kappa}_1^{-1/2})$. The maximal quotient of $\bb{V}_{\omega_1}\otimes_{\Lambda^F} F_{\hg}$ where $T_{\frk{q}}'$ acts by multiplication by $\hg(T_{\frk{q}})$ and $\langle \varpi_{\frk{q}}\rangle'$ acts by multiplication by $\hg(\langle \varpi_{\frk{q}}\rangle)$, for all prime $\frk{q}$ coprime to $N$, is isomorphic to a direct sum of copies of $\As(\bb{V}_\hg)^\dagger(-1)(\chi_g^{-1}\circ \omega_N)$. The choice of a test vector $\breve{\hg}$ for $\hg$ determines a $\Lambda^F[G_\bb{Q}]$-equivariant homomorphism
\[
\bb{V}_{\omega_1}(\bm{\eta}_1)\longrightarrow \As(\bb{V}_\hg)^\dagger(-1)(\chi_g^{-1}\circ\omega_N).
\]
Similarly, on account of Lemma~\ref{lemma:LLZ}, the choice of a test vector $\breve{\hf}$ for $\hf$ determines $\Lambda_{\cl{O}}[G_\bb{Q}]$-equivariant homomorphisms,
\[
\bb{V}_{\omega_0,n}(\bm{\eta}_0)\longrightarrow \Ind_{\cl{K}[n]}^\bb{Q}\Lambda_\cl{O}(\psi_\frk{P}^{-1}\kappa_{\ac}^{k-2}\bm{\kappa}_{\ac}^{-1})(1-k/2)((\chi_\psi^{-1}\epsilon_{\cl{K}})\circ\omega_N)
\]
for each positive integer $n$, and, again by Lemma~\ref{lemma:LLZ}, these homomorphisms can be chosen so that they satisfy the compatibility property stated there. Assume that we have chosen such a family of compatible homomorphisms, and define
\begin{equation}\label{eq:class-test-vect}
\tilde{\kappa}_{\infty,n}(\hf,\hg)\in H^1(\bb{Q}, \Ind_{\cl{K}[n]}^\bb{Q}\Lambda_\cl{O}(\psi_\frk{P}^{-1}\kappa_{\ac}^{k-2}\bm{\kappa}_{\ac}^{-1})(-k/2)\hat{\otimes}_\cl{O} \As(\bb{V}_\hg)^\dagger)
\end{equation}
to be the resulting image of $\kappa_{\infty,n}(\omega_0,\omega_1;\eta)$. By Shapiro's lemma, we can regard $\tilde{\kappa}_{\infty,n}(\hf,\hg)$ as an Iwasawa cohomology class
\[
\tilde{\kappa}_{\infty,n}(\hf,\hg)\in H^1_{\Iw}(\cl{K}[np^\infty],\bb{V}_{\hg,\psi}(\kappa_{\ac}^{k-2}))=\varprojlim_s H^1(\cl{K}[np^s],\bb{V}_{\hg,\psi}(\kappa_{\ac}^{k-2})),
\]
where $\bb{V}_{\hg,\psi}=\As(\bb{V}_\hg)^\dagger(\psi_\frk{P}^{-1})(-k/2)$. Put
\[
\kappa_{\infty,n}(\hf,\hg)=\omega_0(n)\omega_1(n)\kappa(n)^{-1}\chi_{\hg}^{-1}(\varpi_{\frk{n}_1})[\xi_{\varpi_{\frk{n}_1}}^{-1}]\tilde{\kappa}_{\infty,n}(\hf,\hg).
\]

\begin{prop}
    Let $n$ be a squarefree product of primes coprime to $pN$ and split in $F$,  and let $q$ be a rational prime coprime to $pNn$ and which splits in $F$ as $(q)=\frk{q}_1\frk{q}_2$ and splits in $\cl{K}$ as $(q)=\frk{q}\overline{\frk{q}}$. Then,
    \begin{align*}
        \cores_{\cl{K}[nq]/\cl{K}[n]}\kappa_{\infty,n}(\hf,\hg) &= q\omega_1(q)\chi_{\hg}(\varpi_{\frk{q}_1})\chi_{\hg}(\varpi_{\frk{q}_2}) \left(\frac{\kappa_{\ac}^{2-k}\psi_{\frk{P}}(\Fr_{\frk{q}}^{-1})}{q^{k/2}}\Fr_{\frk{q}}^{-1}\right)^2\\
        & -\eta_1(q)\kappa_1^{-1/2}(q)a_{\frk{q}_1}(\hg) a_{\frk{q}_2}(\hg) \left(\frac{\kappa_{\ac}^{2-k}\psi_{\frk{P}}(\Fr_{\frk{q}}^{-1})}{q^{k/2}}\Fr_{\frk{q}}^{-1}\right) \\
        & + \chi_{\hg}(\varpi_{\frk{q}_1})[\xi_{\varpi_{\frk{q}_1}}]a_{\frk{q}_1}(\hg)^2+q^{-1}\chi_{\hg}(\varpi_{\frk{q}_2})[\xi_{\varpi_{\frk{q}_2}}] a_{\frk{q}_2}(\hg)^2-q^{-1}(q^2+1) \\
        &-\eta_1(q)\kappa_1^{-1/2}(q)a_{\frk{q}_1}(\hg) a_{\frk{q}_2}(\hg) \left(\frac{\kappa_{\ac}^{2-k}\psi_{\frk{P}}(\Fr_{\overline{\frk{q}}}^{-1})}{q^{k/2}}\Fr_{\overline{\frk{q}}}^{-1}\right) \\
        & +q\omega_1(q)\chi_{\hg}(\varpi_{\frk{q}_1})\chi_{\hg}(\varpi_{\frk{q}_2}) \left(\frac{\kappa_{\ac}^{2-k}\psi_{\frk{P}}(\Fr_{\overline{\frk{q}}}^{-1})}{q^{k/2}}\Fr_{\overline{\frk{q}}}^{-1}\right)^2.
    \end{align*}
\end{prop}
\begin{proof}
    The corestriction map in Galois cohomology corresponds to the norm map in Lemma~\ref{lemma:LLZ} under Shapiro's isomorphism. Thus, the result follows from Lemma~\ref{lemma:LLZ} and  Proposition~\ref{prop:norm-relations1} after a somewhat tedious computation.
\end{proof}


Similarly as in Section~\ref{subsec:Selmer}, we can define a balanced Selmer group 
\[
\Sel_{\bal}(\Ind_{\cl{K}[n]}^\bb{Q}\Lambda_\cl{O}(\bm{\kappa}_{\ac}^{-1})\hat{\otimes}_\cl{O}\bb{V}_{\hg,\psi})\subset H^1(\bb{Q}^\Sigma/\bb{Q},\Ind_{\cl{K}[n]}^\bb{Q}\Lambda_\cl{O}(\bm{\kappa}_{\ac}^{-1})\hat{\otimes}_\cl{O}\bb{V}_{\hg,\psi}),
\]
and we let $\Sel_{\bal}(\cal{K}[np^\infty],\bb{V}_{\hg,\psi})$  denote its image under the natural isomorphism
\[
H^1(\bb{Q}^\Sigma/\bb{Q},\Ind_{\cl{K}[n]}^\bb{Q}\Lambda_\cl{O}(\bm{\kappa}_{\ac}^{-1})\hat{\otimes}_\cl{O}\bb{V}_{\hg,\psi})\cong 
H^1_{\Iw}(\cal{K}[np^\infty],\mathbb V_{\hg,\psi})
\]
arising from Shapiro's lemma. 

It is now easy to deduce from the above results that we thus obtain an anticyclotomic Euler system for $\bb{V}_{\hg,\psi}$ in the sense of \cite{JNS}, as summarized in the next theorem. 

\begin{theorem}\label{thm:ES}
Let $\mathcal{S}$ be the set of squarefree products of primes $q$ split in both $F$ and $\cl{K}$ and coprime to $pN$. There exists a collection of classes
\[
\left\lbrace\kappa_{\psi,\hg,n,\infty}\in \Sel_{\bal}(\cl{K}[np^\infty],\bb{V}_{\hg,\psi})\;\colon\; n\in\mathcal{S}\right\rbrace
\]
such that whenever $n, nq\in\mathcal{S}$ with $q$ a prime, we have
\begin{equation}
\cores_{\cl{K}[nq]/\cl{K}[n]}(\kappa_{\psi,\hg,nq,\infty})=P_{\mathfrak{q}}(\bb{V}_{\hg,\psi};\Fr_{\mathfrak{q}}^{-1})\,\kappa_{\psi,\hg,n,\infty}, \nonumber
\end{equation}
where $\frk{q}$ is a prime of $\cl{K}$ above $q$, and 
\[
P_{\frk{q}}(\bb{V}_{\hg,\psi};X)=\det(1-\Fr_{\frk{q}}^{-1}X\vert\bb{V}_{\hg,\psi}).
\]
In particular, we obtain an anticyclotomic Euler system for each specialization of $\hg$.
\end{theorem}

\begin{proof} 
One can check that, defining
\[
\kappa_{\infty,n}^\ast(\hf,\hg):=\prod_{q\mid n} \omega_1^{-1}(q)\chi_{\hg}^{-1}(\varpi_{\frk{q}_1})\chi_{\hg}^{-1}(\varpi_{\frk{q}_2})\left(\kappa_{\ac}^{2-k}\psi_{\frk{P}}(\Fr_{\overline{\frk{q}}})\Fr_{\overline{\frk{q}}}\right)^2 \kappa_{\infty,n}(\hf,\hg)
\]
for all $n\in\mathcal{S}$, from the norm relations of Proposition~\ref{prop:norm-relations1} one has
\[
\cores_{\cl{K}[nq]/\cl{K}[n]}\kappa_{\infty,nq}^\ast (\hf,\hg)\equiv P_{\frk{q}}(\bb{V}_{\hg,\psi}(\kappa_{\ac}^{k-2});\Fr_{\frk{q}}^{-1})\kappa_{\infty,n}^\ast (\hf,\hg) \pmod{q-1}
\]
for any rational prime $q$ with $nq\in\mathcal{S}$, where 
\[
P_{\frk{q}}(\bb{V}_{\hg,\psi}(\kappa_{\ac}^{k-2});X)=\det(1-\Fr_{\frk{q}}^{-1}X\vert\bb{V}_{\hg,\psi}(\kappa_{\ac}^{k-2})).
\]
By \cite[Lem.~9.6.1]{Rub}, from $\kappa_{\infty,n}^\ast (\hf,\hg)$ we can obtain classes satisfying  norm relations involving the Euler factor $P_{\frk{q}}(\bb{V}_{\hg,\psi}(\kappa_{\ac}^{k-2});\Fr_{\frk{q}}^{-1})$, and by the twisting result of \cite[Thm.~6.3.5]{Rub}, we can get rid of the twist by $\kappa_{\ac}^{k-2}$, yielding classes $\kappa_{\psi,\hg,n,\infty}\in H^1_{\Iw}(\cl{K}[np^\infty],\bb{V}_{\hg,\psi})$ with the stated norm relations. Finally, the fact that these classes belong to the balanced Selmer group follows from \cite[Thm.~5.9]{NN16} by the same argument as in \cite[Prop.~4.9]{ACR}.
\end{proof}


\section{Verifying the hypotheses}\label{sec:big-image}

Let $g$ and $\psi$ be as in the introduction, and let $L$ be a finite extension of $\bb{Q}$ containing $F$, the Fourier coefficients of $g$, and the image of $\psi$. For each prime $\frk{P}$ of $L$, let
\[
\rho_{g,\frk{P}}^{\std}: G_F \longrightarrow \mathrm{GL}_2(L_{\frk{P}})
\]
be the corresponding Galois representation attached to $g$, as described in \cite[Thm.~4.3.1]{LLZ-AF}. From these representations, we can define the corresponding Asai representations
\[
\rho_{g,\frk{P}}: G_{\bb{Q}}\longrightarrow \mathrm{GL}_4(L_{\frk{P}})
\]
by taking the tensor-induction of $\rho_{g,\frk{P}}^{\std}$. 
We denote by $\As(g)_{\frk{P}}$ the corresponding $4$-dimensional Galois-module. Since we are interested in prescribing good choices of $\frk{P}$, we write $V_{\frk{P}}$ for the $G_{\cl{K}}$-representation
\[
V_{\frk{P}} = \As(g)_{\frk{P}}(\psi_{\mathfrak P}^{-1})(2-l-k/2).
\]
within which we fix a Galois stable lattice $T_{\frk{P}}$. Note that the representation $V_{\frk{P}}$ is conjugate self-dual, i.e. $V_{\frk{P}}^c\simeq V_{\frk{P}}^\vee(1)$.

The aim of this section is to give conditions 
under which hypothesis (HS) in \cite[\S8.1]{ACR} hold for $T_{\frk{P}}$ (and so also the weaker hypothesis (HW) for $V_{\frk{P}}$ in \emph{loc.\,cit.}), so the general results of \cite{JNS} can be applied to the Euler system constructed in Theorem~\ref{thm:ES}.

We denote by $g^\sigma$ the internal conjugate of $g$. Throughout this section, we make the following assumptions:
\begin{itemize}
    \item $g$ is not of CM type;
    \item $g$ is not Galois-conjugate to $g^\sigma$.
\end{itemize}


As in \cite[\S 3.1]{Loe}, we define subfields $F_g=F_{g^\sigma}$ of $L$, open subgroups $H_g$ and $H_{g^\sigma}$ of $G_{F}$, quaternion algebras $B_g$ and $B_{g^\sigma}$, and algebraic groups $G_g$ and $G_{g^\sigma}$, and we put
\[
B=B_g\times B_{g^\sigma}, \quad G=G_g\times_{\bb{G}_m} G_{g^\sigma}.
\]
Let $H = H_g \cap H_{g^\sigma}\cap G_{\cl{K}(\mathfrak c)^{\circ}}$, where $\cl{K}(\mathfrak c)^\circ$ denotes the ray class field of $\cl{K}$ of conductor $\frk{c}$. Then we have an adelic representation
\[
\tilde{\rho}_{g,g^\sigma}^{\std} : H \longrightarrow G(\hat{\bb{Q}})
\]
and representations
\[
\tilde{\rho}_{g,g^\sigma,p}^{\std} : H \longrightarrow G(\bb{Q}_p)
\]
for every rational prime $p$. According to \cite[Thm. 3.2.2]{Loe}, for all but finitely many primes $p$ we can conjugate $\tilde{\rho}_{g,,g^\sigma,p}^{\std}$ so that $\tilde{\rho}_{g, g^\sigma,p}^{\std}(H) = G(\mathbb Z_p)$. Note that although the result is stated in the setting of elliptic modular forms, it follows from [\emph{op.\,cit.}, Rem.~2.3.2] that the same result holds in the setting of Hilbert modular forms.



\begin{defn}
Let $\frk{P}$ be a prime of $L$ above $p$, and let $E$ be the subfield of $L_{\frk{P}}$ generated by the field $F_{g}$. We say that the prime $\frk{P}$ is {\it good} if the following conditions hold:
\begin{itemize}
    \item $p \geq 7$;
    \item $p$ is unramified in $B$;
    \item $p$ is coprime to $\frk{c}$ and $N_g$;
    \item $\tilde{\rho}_{g,g^\sigma,p}^{\std}(H) = G(\mathbb Z_p)$;
    \item $E = \mathbb Q_p$.
\end{itemize}
\end{defn}


\begin{lemma}\label{lemma:big-image}
Assume that there is at least one prime which divides $D_{\cl{K}}$ but not $D_F N_g$. Then, if $\frk{P}$ is a good prime,
\[
(\rho_{g,\frk{P}}\times\rho_{g^\sigma,\frk{P}})(H\cap G_{\cl{K}(p^\infty)^\circ})=\mathrm{SL}_2(\bb{Z}_p)\times \mathrm{SL}_2(\bb{Z}_p).
\]
\end{lemma}
\begin{proof}
For the ease of notation, write $\cal{K}_\infty$ for the anticyclotomic $\bb{Z}_p$-extension $\cal{K}[p^\infty]$ of $\cal{K}$. Let $F(\rho_{g,\frk{P}})$ and $F(\rho_{g^\sigma,\frk{P}})$ be the Galois extensions of $F$ cut out by the representations $\rho_{g,\frk{P}}$ and $\rho_{g^\sigma,\frk{P}}$ attached to $g$ and $g^\sigma$, respectively, and note that they are unramified outside $pD_F N_g$. Therefore, the condition on $D_{\cl{K}}$ implies that $F\cl{K}\cap F(\rho_{g,\frk{P}})=F$ and $F\cl{K}\cap F(\rho_{g^\sigma,\frk{P}})=F$. Moreover, since any Galois extension of $F$ contained in $F\cl{K}_\infty$ must itself contain $F\cl{K}$, we also have $F\cl{K}_\infty\cap F(\rho_{g,\frk{P}})=F$ and $F\cl{K}_\infty\cap F(\rho_{g^\sigma,\frk{P}})=F$.

The conditions on $\frk{P}$ imply that
\[
(\rho_{g,\frk{P}}\times\rho_{g^\sigma,\frk{P}})(H\cap G_{\bb{Q}(\mu_{p^\infty})})=\mathrm{SL}_2(\bb{Z}_p)\times \mathrm{SL}_2(\bb{Z}_p),
\]
and, from the remarks in the previous paragraph, it follows that
\[
(\rho_{g,\frk{P}}\times\rho_{g^\sigma,\frk{P}})(H\cap G_{\cl{K}_\infty(\mu_{p^\infty})})=\mathrm{SL}_2(\bb{Z}_p)\times \mathrm{SL}_2(\bb{Z}_p).
\]
Finally, since $H\cap G_{\cl{K}(p^\infty)^\circ}$ is a normal subgroup of $H\cap G_{\cl{K}_\infty(\mu_{p^\infty})}$ of index dividing $p-1$ and there are no such subgroups in $\mathrm{SL}_2(\bb{Z}_p)\times \mathrm{SL}_2(\bb{Z}_p)$, the lemma follows.
\end{proof}

Now we are able to give conditions on $\frk{P}$ under which the results of \cite{JNS} can be applied to our setting, i.e. to the representation $T_{\frk{P}}$ defined above. We say that $g$ is \emph{non-Eisenstein} (resp. \emph{$p$-distinguished}) at $\frk{P}$ if the residual representation $\bar{\rho}_{g,\frk{P}}$ associated to $\rho_{g,\frk{P}}$ is irreducible (resp. the semi-simplication of $\bar{\rho}_{g,\frk{P}}\vert_{G_{F_{\frk{p}}}}$ is non-scalar for every prime $\frk{p}$ of $F$ above $p$).

\begin{prop}\label{prop:suff}
Assume that there is at least one prime which divides $D_{\cl{K}}$ but not $D_FN_g$. Let $\frk{P}$ be a good prime above $p$ at which $g$ is non-Eisenstein and $p$-distinguished, and assume that there exists $\eta\in G_{F\cl{K}(p^\infty)^\circ}$ such that $\psi_{\frk{P}}(\eta)\neq \psi_{\frk{P}}^c(\eta)$ modulo $p$. Then hypothesis (HS) in \cite[\S{8.1}]{ACR} hold for $T_\frk{P}$.
\end{prop}
\begin{proof}
Since $\psi_{\frk{P}}$
is trivial when restricted to $H\cap G_{\cl{K}(p^\infty)^\circ}$, condition (1') follows easily from the previous Lemma~\ref{lemma:big-image}. To prove condition (2'), we closely follow the proof of \cite[Prop. 4.2.1]{Loe}. We regard $\chi_g$ as a finite order character of $G_F$ using class field theory. By Lemma~\ref{lemma:big-image}, the image of $\eta H\cap G_{\cl{K}(p^\infty)^\circ}$ under $\rho_{g,\frk{P}}\times\rho_{g^\sigma,\frk{P}}$ contains all the elements of the form
\[
\left(\begin{pmatrix} x & 0 \\ 0 & x^{-1}\chi_g(\eta)\end{pmatrix},\begin{pmatrix} y & 0 \\ 0 & y^{-1}\chi_g^\sigma(\eta)\end{pmatrix}\right),\quad x,y\in \bb{Z}_p^\times.
\]
Now choose $x\in\bb{Z}_p^\times$ satisfying 
\[
x^{-2}\chi_g(\eta)\not\equiv 1 \pmod p,\quad x^2\chi_g^\sigma(\eta)\psi_{\frk{P}}(\eta)^{-2}\not\equiv 1 \pmod p,
\]
which is possible since $p\geq 7$, and let $y=x^{-1}\psi_{\frk{P}}(\eta)$. Choose $\sigma_0\in \eta H\cap G_{\cl{K}(p^\infty)^\circ}$ whose image under $\rho_{g,\frk{P}}\times\rho_{g^\sigma,\frk{P}}$ is given by the element above, with the choices of $x$ and $y$ which we have just specified. Then, the eigenvalues of $\sigma_0$ acting on $T_\frk{P}$ are 1, $x^{-2}\chi_g(\eta)$, $x^2\chi_g^\sigma(\eta)\psi_{\frk{P}}(\eta)^{-2}$ and $\psi_{\frk{P}}^c(\eta)\psi_{\frk{P}}(\eta)^{-1}$, which proves condition (2').

To check condition (3'), we can argue as in \cite[Rem.~11.1.3]{KLZ}. By the previous lemma, we can find an element $\tau_0\in H\cap G_{\cl{K}(p^\infty)^\circ}$ such that
\[
(\rho_{g,\frk{P}}\times\rho_{g^\sigma,\frk{P}})(\tau_0)=\left(\begin{pmatrix} -1 & 0 \\ 0 & -1\end{pmatrix},\begin{pmatrix} 1 & 0 \\ 0 & 1\end{pmatrix}\right),
\]
so $\tau_0$ acts on $T_\frk{P}$ as multiplication by $-1$. Finally, condition (4') follows from the assumption that $g$ is non-Eisenstein and $p$-distinguished.
\end{proof}

\section{Applications}

In this section we apply the general machinery of \cite{JNS} to the Euler system constructed in this paper to  deduce applications towards the Bloch--Kato conjecture and the anticyclotomic Iwasawa Main Conjecture for Asai representations. 




\subsection{Bloch--Kato conjecture}

As in the introduction, let $g$ be a Hilbert eigenform of parallel weight $l$ and let $\psi$ be a Hecke character of infinity type $(1-k,0)$ for some even integer $k \geq 2$, and put
\[
V:=\mathrm{As}(V_g)\vert_{G_{\cal{K}}}(\psi_{\frk{P}}^{-1})(2-l-k/2),
\]
Write $\kappa_{\psi,g,h,n}\in H^1(\cal{K}[n],T)$ for the image of the Iwasawa cohomology classes of Theorem~\ref{thm:ES} under the corestriction map $H^1_{\Iw}(\cal{K}[np^\infty],T)\rightarrow H^1(\cal{K}[n],T)$. 

The following result yields evidence towards the Bloch--Kato conjecture in rank 1. 

\begin{theorem}\label{thm:BK}
Assume that $p$ splits in $\cal{K}$ and $p\nmid h_{\cal{K}}$, and that $g$ is ordinary at $p$. Let $\frk{P}$ be a good prime of $L$ above $p$ at which $g$ is non-Eisenstein and $p$-distinguished, and assume that the conditions in Proposition~\ref{prop:suff} are satisfied.
Let 
\[ \kappa_{\psi,g} := \kappa_{\psi,g,1}. 
\] 
If $k<2l$, then the following implication holds:
\[
\kappa_{\psi,g} \neq 0\quad\Longrightarrow\quad\dim_E \Sel(\cl{K}, V) =1,
\]
where $\Sel(\cl{K}, V) \subset H^1(\cl{K}, V)$ is the Bloch--Kato Selmer group.
\end{theorem}

\begin{proof}
By virtue of Theorem~\ref{thm:ES}, the classes $\kappa_{\psi,g,h,n}$ land in  $\Sel_{\bal}(\cal{K}[n],T)$  and (for varying $n\in\mathcal{S}$) form an anticyclotomic Euler system for $V$. 
Thus, from the general results of \cite{JNS}, as recalled in \cite[Thm.~8.3]{ACR}, and whose technical conditions are satisfied thanks to Proposition~\ref{prop:suff}, the non-vanishing of $\kappa_{\psi,g}$ implies that the balanced Selmer group  $\Sel_{\bal}(\cal{K},V)$ is $1$-dimensional over $E$. Since for $k<2l$ one can show the equality
\[
\Sel_{\bal}(\cal{K},V)=\Sel(\cal{K},V)
\]
following the argument in \cite[Lem.~9.1]{ACR}, this concludes the proof.
\end{proof}


\subsection{Iwasawa Main Conjecture}

With notations as in Section~\ref{subsec:Selmer}, put 
\[
\mathbb A_{\hf g}^\dag={\Hom}_{\mathbb Z_p}(\mathbb V_{\hf g}^\dag,\mu_{p^\infty}),
\]
and for $\mathcal{L}\in\{\bal,\mathcal{F}\}$, define the Selmer group ${\Sel}_{\mathcal L}(\mathbb A_{\hf g}^\dag)$ 
as in Definition~\ref{def:Sel}, 
taking $H^1_{\mathcal L}(\mathbb Q_p,\mathbb A_{\hf g}^\dag)$ to be the orthogonal complement of $H^1_{\mathcal L}(\mathbb Q_p,\mathbb V_{\hf g}^\dag)$ under the local Tate duality
\[
H^1(\mathbb Q_p,\mathbb V_{\hf g}^\dag)\times H^1(\mathbb Q_p,\mathbb A_{\hf g}^\dag)\longrightarrow\bb{Q}_p/\bb{Z}_p.
\]
Taking $\hf$ to be the CM Hida family attached to $\psi$ as in Section~\ref{subsec:CM-hida}, similarly as in Lemma~\ref{lemma:LLZ} we then have a $G_{\bb{Q}}$-module isomorphism
\[
\mathbb V_{\hf g}^\dagger\simeq\Lambda_{\mathcal O}(\bm{\kappa}_{\text{ac}}^{-1})\hat{\otimes}_{\cl{O}} T.
\]
We then let $\Sel_{\mathcal{L}}(\cal{K}[p^\infty],T)\subset H^1_{\Iw}(\cal{K}[p^\infty],T)$ be the Selmer group corresponding to $\Sel_{\mathcal{L}}(\mathbb{V}_{\hf g}^\dagger)$ under the Shapiro isomorphism
\[
H^1(\mathbb Q,\mathbb V_{\hf g}^\dag) \cong H^1(\cl{K},\Lambda_{\mathcal O}(\bm{\kappa}_{\text{ac}}^{-1})\hat{\otimes}_{\cl{O}} T)\cong H^1_{\Iw}(\cl{K}_\infty, T),
\] 
and likewise let $\Sel_{\mathcal{L}}(\cal{K}[p^\infty],A)$ 
correspond to $\Sel_{\mathcal{L}}(\mathbb{A}_{\hf g}^\dagger)$ under $H^1(\mathbb Q,\mathbb V_{\hf g}^\dag)\cong H^1(\cal{K}[p^\infty],A)$. 
Finally, put
\[
X_{\mathcal L}(\cal{K}[p^\infty],A):=\Hom_{\text{cont}}(\Sel_{\mathcal L}(\cal{K}[p^\infty],A), \mathbb Q_p/\mathbb Z_p)
\]
for the Pontryagin dual of $\Sel_{\mathcal L}(\cal{K}[p^\infty],A)$, which is a finitely generated $\Lambda_\hf$-module. 




\begin{theorem}\label{thm:IMC}
Let the hypotheses as in Theorem \ref{thm:BK}. If the class
\[
\kappa_{\psi,g,\infty}:=\kappa_{\psi,g,1,\infty}
\]
is not $\Lambda_\hf$-torsion, then the modules $\Sel_{\bal}(\cal{K}[p^\infty],T)$ and $X_{\bal}(\cal{K}[p^\infty],A)$ have both $\Lambda_{\hf}$-rank one, and 
\[
\Char_{\Lambda_{\hf}}(X_{\bal}(\cal{K}[p^\infty],A)_{\tors})\supset \Char_{\Lambda_{\hf}}\biggl(\frac{\Sel_{\bal}(\cal{K}[p^\infty],T)}{\Lambda_{\hf} \cdot\kapinftys} \biggr)^2
\]
as ideals in $\Lambda_{\hf}\otimes_{\bb{Z}_p}\bb{Q}_p$, where the subscript $\mathrm{tors}$ denotes the $\Lambda_{\hf}$-torsion submodule.
\end{theorem}

\begin{proof}
In view of Proposition~\ref{prop:suff}, this follows from the general results of \cite{JNS}, in the form stated in \cite[Thm.~8.5]{ACR}, applied to the Euler system of Theorem~\ref{thm:ES}.
\end{proof}

The prediction that equality should hold in Theorem~\ref{thm:IMC}, even as ideals in $\Lambda_\hf$, is suggested by a natural extension of Perrin-Riou's Heegner point main conjecture \cite[Conj.~B]{PR}. It also admits a natural counterpart expressing the characteristic ideal of the $\cal{F}$-unbalanced Selmer group $X_{\mathcal F}(\cal{K}[p^\infty],A)$ in terms of a $p$-adic Asai $L$-function. This will be explored in the sequel to this work.

\bibliographystyle{amsalpha}
\bibliography{refs}

\end{document}